\documentclass[a4paper,11pt,oneside]{amsart}
\usepackage[left=2.7cm,right=2.7cm,top=3.5cm,bottom=3cm]{geometry}

\usepackage[all]{xy}
\usepackage{amscd}
\usepackage{amsfonts}
\usepackage{amssymb}
\usepackage{amsmath}
\usepackage{latexsym}

\numberwithin{equation}{section}

\newtheorem{thm}{Theorem}[section]
\newtheorem{cor}[thm]{Corollary}
\newtheorem{lem}[thm]{Lemma}
\newtheorem{prop}[thm]{Proposition}

\theoremstyle{definition}
\newtheorem{df}[thm]{Definition}
\newtheorem{rem}[thm]{Remark}
\newtheorem{rems}[thm]{Remarks}
\newtheorem{ass}[thm]{Assumption}

\newcommand{\rank}{\text{rank}}

\newcommand{\field}[1]{\mathbb{#1}}
\newcommand{\Q}{\field{Q}}
\newcommand{\C}{\field{C}}
\newcommand{\R}{\field{R}}
\newcommand{\N}{\field{N}}
\newcommand{\Z}{\field{Z}}

\newcommand{\F}{\field{F}}

\newcommand{\Gm}{\Gamma}
\renewcommand{\L}{\Lambda}
\newcommand{\dc}{{\bf C}_{\infty}}
\newcommand{\cp}{{\bf C}_{\mathfrak p}}

\newcommand{\Div}{\mathcal Div}

\newcommand{\cale}{\mathcal E}

\newcommand{\ad}{{\bf A}}
\newcommand{\idl}{{\bf I}}

\newcommand{\liminv}{\displaystyle \lim_{\leftarrow}}
\newcommand{\limdir}{\displaystyle \lim_{\rightarrow}}
\newcommand{\ol}{\mathcal O}
\newcommand{\cala}{\mathcal A}
\newcommand{\calm}{\mathcal M}
\newcommand{\calk}{\mathcal K}

\newcommand{\call}{\mathcal L}

\newcommand{\cals}{\mathcal S}
\newcommand{\calf}{\mathcal F}
\newcommand{\calx}{\mathcal X}
\newcommand{\caln}{\mathcal N}
\newcommand{\calc}{\mathcal C}
\newcommand{\calu}{\mathcal U}
\newcommand{\pr}{\mathfrak p}
\newcommand{\gotn}{\mathfrak n}

\newcommand{\sri}{\twoheadrightarrow}
\newcommand{\iri}{\hookrightarrow}
\newcommand{\ri}{\rightarrow}

\newcommand{\ov}{\overline}
\newcommand{\wt}{\widetilde}
\newcommand{\wh}{\widehat}
\newcommand{\dl}[1]{\lim_{\buildrel \longrightarrow\over{#1}}}
\newcommand{\il}[1]{\lim_{\buildrel \longleftarrow\over{#1}}}
\newfont{\cyr}{wncyr10 scaled 1000}

\newcommand{\dlog}{{\rm dlog\,}}
\newcommand{\End}{\operatorname{End}}

\title[Iwasawa theory over function fields]{Aspects of Iwasawa theory over function fields}

\author[A.~Bandini, F.~Bars, I.~Longhi]{Andrea Bandini, Francesc Bars*\protect\footnote{* Supported by MTM2013-40680-P} ,
 Ignazio Longhi **\protect\footnote{** Supported by NSC 099-2811-M-002-096}}

\begin{document}

\maketitle

\begin{abstract}
We consider $\Z_p^{\mathbb{N}}$-extensions $\calf$ of a global function field $F$ and study various aspects of Iwasawa theory with emphasis
on the two main themes already (and still) developed in the number fields case as well. When dealing with the Selmer group of an abelian
variety $A$ defined over $F$, we provide all the ingredients to formulate an Iwasawa Main Conjecture relating the Fitting ideal and the
$p$-adic $L$-function associated to $A$ and $\calf$. We do the same, with characteristic ideals and $p$-adic $L$-functions, in the case
of class groups (using known results on characteristic ideals and Stickelberger elements for $\Z_p^d$-extensions). The final section
provides more details for the cyclotomic $\Z_p^{\mathbb{N}}$-extension arising from the torsion of the Carlitz module: in particular,
we relate cyclotomic units with Bernoulli-Carlitz numbers by a Coates-Wiles homomorphism.
\end{abstract}

{\small \noindent{\bf 2010 MSB:} Primary: 11R23. Secondary: 11R58,
11G05, 11G10, 14G10, 11R60.

\noindent{\bf Keywords:} Iwasawa Main Conjecture, global function
fields, $L$-functions, Selmer groups, class groups,
Bernoulli-Carlitz numbers}

\section{Introduction}

The main theme of number theory (and, in particular, of arithmetic geometry) is probably the study of representations of the Galois
group $Gal(\overline\Q/\Q)$ - or, more generally, of the absolute Galois group $G_F:=Gal(F^{sep}/F)$ of some global field $F$. A basic
philosophy (basically, part of the yoga of motives) is that any object of arithmetic interest is associated with a $p$-adic realization,
which is a $p$-adic representation $\rho$ of $G_F$ with precise concrete properties (and to any $p$-adic representation with such properties
should correspond an arithmetic object). Moreover from this $p$-adic representation one defines the $L$-function associated to the arithmetic
object. Notice that the image of $\rho$ is isomorphic to a compact subgroup of $GL_n(\Z_p)$ for some $n$, hence it is a $p$-adic Lie group and
the representation factors through $Gal(\calf'/F)$, where $\mathcal F'$ contains subextensions $\mathcal F$ and $F'$ such that $\mathcal F/F'$
is a pro-$p$ extension and $F'/F$ and $\calf'/\calf$ are finite.\bigskip

Iwasawa theory offers an effective way of dealing with various issues arising in this context, such as the variation of arithmetic structures
in $p$-adic towers, and is one of the main tools currently available for the knowledge (and interpretation) of zeta values associated to an
arithmetic object when $F$ is a number field \cite{katoicm}. This theory constructs some sort of elements, called $p$-adic $L$-functions,
which provide a good understanding of both the zeta values and the arithmetic properties of the arithmetic object. In particular, the various
forms of Iwasawa Main Conjecture provide a link between the zeta side and the arithmetic side.

The prototype is given by the study of class groups in the cyclotomic extensions $\Q(\zeta_{p^n})/\Q$. In this case the arithmetic side
corresponds to a torsion $\L$-module $X$, where $\L$ is an Iwasawa algebra related to $Gal(\Q(\zeta_{p^\infty})/\Q)$ and $X$ measures the
$p$-part of a limit of $cl(\Q(\zeta_{p^n}))$. As for the zeta side, it is represented by a $p$-adic version of the Riemann zeta function,
that is, an element $\xi\in\L$ interpolating the zeta values. One finds that $\xi$ generates the characteristic ideal of $X$.

For another example of Iwasawa Main Conjecture, take $E$ an elliptic curve over $\Q$ and $p$ a prime of good ordinary reduction (in terms
of arithmetic objects, here we deal with the Chow motive $h^1(E)$, as before with $h^0(\Q)$). Then on the arithmetic side the torsion Iwasawa
module $X$ corresponds to the Pontrjagin dual of the Selmer group associated to $E$ and the $p$-adic $L$-function of interest here is an
element $L_p(E)$ in an Iwasawa algebra $\L$ (that now is $\Z_p[[Gal(\Q(\zeta_{p^\infty})/\Q)]]$) which interpolates twists of the $L$-function
of $E$ by Dirichlet characters of $(\Z/p^n)^*$. As before, conjecturally $L_p(E)$ should be the generator of the characteristic ideal of $X$.

In both these cases, we had $F=\Q$ and $\calf'=\Q(\zeta_{p^\infty})$. Of course there is no need for such a limitation and one can take as
$\calf'$ any $p$-adic extension of the global field $F$: for example one can deal with $\Z_p^{n}$-extensions of $F$. A more recent creation
is non-commutative Iwasawa theory, which allows to deal with non-commutative $p$-adic Lie group, as the ones appearing from non-CM elliptic
curves (in particular, this may include the extensions where the $p$-adic realization of the arithmetic object factorizes).\bigskip

In most of these developments, the global field $F$ was assumed to be a number field. The well-known analogy with function fields suggests
that one should have an interesting Iwasawa theory also in the characteristic $p$ side of arithmetic. So in the rest of this paper $F$ will
be a global function field, with $char(F)=p$ and constant field $\F_F$. Observe that there is a rich and well-developed theory of cyclotomic
extension for such an $F$, arising from Drinfeld modules: for a survey on its analogy with the cyclotomic theory over $\Q$ see \cite{thak}.

We shall limit our discussion to abelian Galois extension of $F$. One has to notice that already with this assumption, an interesting
new phenomenon appears: there are many more $p$-adic abelian extensions than in the number field case, since local groups of units are
$\Z_p$-modules of infinite rank. So the natural analogue of the $\Z_p$-extension of $\Q$ is the maximal $p$-adic abelian extension
$\calf/F$ unramified outside a fixed place and we have $\Gamma=Gal(\calf/F)\simeq\Z_p^\N$. It follows that the ring $\Z_p[[\Gamma]]$
is not noetherian; consequently, there are some additional difficulties in dealing with $\Lambda$-modules in this case. Our proposal
is to see $\Lambda$ as a limit of noetherian rings and replace characteristic ideals by Fitting ideals when necessary.

As for the motives originating the Iwasawa modules we want to study, we start considering abelian varieties over $F$ and ask the same
questions as in the number field case. Here the theory seems to be rich enough. In particular, various control theorems allow to define the
algebraic side of the Iwasawa Main conjecture. As for the analytic part, we will sketch how a $p$-adic $L$-function can be defined for
modular abelian varieties.

Then we consider the Iwasawa theory of class groups of abelian extensions of $F$. This subject is older and more developed: the Iwasawa
Main Conjecture for $\Z_p^n$-extension was already proved by Crew in the 1980's, by geometric techniques. We concentrate on
$\Z_p^\N$-extensions, because they are the ones arising naturally in the cyclotomic theory; besides they are more naturally related to
characteristic $p$ $L$-functions (a brave new world where zeta values have found another, yet quite mysterious, life). The final section,
which should be taken as a report on work in progress, provides some material for a more cyclotomic approach to the Main Conjecture.

\subsection{Contents of the paper} In section \ref{ConTh} we study the structure of Selmer groups associated with elliptic curves (and, more
in general, with abelian varieties) and $\Z_p^d$-extensions of a global function field $F$. We use the different versions of control theorems
avaliable at present to show that the Pontrjagin duals of such groups are finitely generated (sometimes torsion) modules over the appropriate
Iwasawa algebra. These results allow us to define characteristic and Fitting ideals for those duals. In section \ref{LamMod}, taking the
$\Z_p^d$-extensions as a filtration of a $\Z_p^{\mathbb{N}}$-extension $\calf$, we can use a limit argument to define a (pro-)Fitting ideal
for the Pontrjagin dual of the Selmer group associated with $\calf$. This (pro-)Fitting ideal (or, better, one of its generators) can be
considered as a worthy candidate for an algebraic $L$-function in this setting.
In section \ref{s:modGL2} we deal with the analytic counterpart, giving a brief description of the $p$-adic $L$-functions which have been defined
(by various authors) for abelian varieties and the extensions $\calf/F$. Sections \ref{LamMod} and \ref{s:modGL2} should provide the ingredients
for the formulation of an Iwasawa Main Conjecture in this setting.
In section \ref{s:classgroups} we move to the problem of class groups. We use some techniques of an (almost) unpublished work of Kueh, Lai and Tan
to show that the characteristic ideals of the class groups of $\Z_p^d$-subextensions of a cyclotomic $\Z_p^{\mathbb{N}}$-extension are generated
by some Stickelberger element. Such a result can be extended to the whole $\Z_p^{\mathbb{N}}$-extension via a limit process because, at least
under a certain assumption, the characteristic ideals behave well with respect to the inverse limit (as Stickelberger elements do).
This provides a new approach to the Iwasawa Main Conjecture for class groups. At the end of section \ref{s:classgroups} we briefly recall
some results on what is known about class groups and characteristic $p$ zeta values.
Section \ref{s:carlitz} is perhaps the closest to the spirit of function field arithmetic. For simplicity we deal only with the Carlitz module.
We study the Galois module of cyclotomic units by means of Coleman power series and show how it fits in an Iwasawa Main Conjecture. Finally we
compute the image of cyclotomic units by Coates-Wiles homomorphisms: one gets special values of the Carlitz-Goss zeta function, a result which
might provide some hints towards its interpolation\,\footnote{A different approach using a version of Iwasawa Main Conjecture for the cyclotomic
Carlitz extension and leading to information on special values of the Carlitz-Goss zeta function is carried out in \cite{ABBL}.}.\bigskip

This paper was first written in 2010 so it reflects the situation at the time. We have added a few references to more recent
developments related with the theory presented here but have not attempted to include detailed descriptions of new results.
A recent excellent source for the arithmetic of function fields is the book \cite{BLTBook}, in particular (since here we focus
on Iwasawa theory) we mention the paper \cite{BT} which also covers the non-commutative approach.

\subsection{Some notations} \label{notation1} Given a field $L$, $\overline L$ will denote an algebraic closure and $L^{sep}$ a separable closure;
we shall use the shortening $G_L:=Gal(L^{sep}/L)$. When $L$ is (an algebraic extension of) a global field, $L_v$ will be its completion at the
place $v$, $\ol_v$ the ring of integers of $L_v$ and $\F_v$ the residue field. We are going to deal only with global fields of positive
characteristic: so $\F_L$ shall denote the constant field of $L$.

As mentioned before, let $F$ be a global field of characteristic $p>0$, with field of constants $\F_F$ of cardinality $q$. We also
fix algebraic closures $\overline F$ and $\overline{F_v}$ for any place $v$ of $F$, together with embeddings
$\overline F\hookrightarrow\overline{F_v}$, so to get restriction maps $G_{F_v}\hookrightarrow G_F$. All algebraic extensions of $F$
(resp. $F_v$) will be assumed to be contained in $\overline F$ (resp. $\overline{F_v}$).

Script letters will denote infinite extensions of $F$. In particular, $\calf$ shall always denote a Galois extension of $F$, ramified only at
a finite set of places $S$ and such that $\Gm:=Gal(\calf/H)$ is a free $\Z_p$-module, with $H/F$ a finite subextension (to ease notations,
in some sections we will just put $H=F$); the associated Iwasawa algebra is $\L:=\Z_p[[\Gm]]$. We also put $\tilde\Gm:=Gal(\calf/F)$ and
$\tilde\L:=\Z_p[[\tilde\Gm]]$.

The Pontrjagin dual of an abelian group $A$ shall be denoted as $A^\vee$.

\begin{rem} Class field theory shows that, in contrast with the number field case, in the characteristic $p$ setting $Gal(\calf/F)$ (and hence
$\Gamma$) can be very large indeed. Actually, it is well known that for every place $v$ the group of 1-units $\ol_{v,1}^*\subset F_v^*$ (which
is identified with the inertia subgroup of the maximal abelian extension unramified outside $v$) is isomorphic to a countable product of copies
of $\Z_p$: hence there is no bound on the dimension of $\Gamma$. Furthermore, the only  $\Z_p^{finite}$-extension of $F$ which arises
somewhat naturally is the arithmetic one $\calf^{arit}$, i.e., the compositum of $F$ with the maximal pro-$p$-extension of $\F_F$. This justifies
our choice to concentrate on the case of a $\Gamma$ of infinite rank: $\calf$ shall mostly be the maximal abelian extension unramified outside
$S$ (often imposing some additional condition to make it disjoint from $\calf^{arit}$).

We also recall that a $\Z_p$-extension of $F$ can be ramified at infinitely many places \cite[Remark 4]{GC}: hence our condition on $S$ is a
quite meaningful restriction.
\end{rem}

\section{Control theorems for abelian varieties}\label{ConTh}

\subsection{Selmer groups}\index{Selmer group} \label{ss:selmer}
Let $A/F$ be an abelian variety, let $A[p^n]$ be the group scheme of $p^n$-torsion points and put $A[p^\infty]:=\limdir A[p^n]$. Since we work
in characteristic $p$ we define the Selmer groups via flat cohomology of group schemes. For any finite algebraic extension $L/F$ let $X_L:=Spec\, L$
and for any place $v$ of $L$ let $L_v$ be the completion of $L$ at $v$ and $X_{L_v}:=Spec\, L_v\,$. Consider the local Kummer embedding
\[ \kappa_{L_v}\,:\, A(L_v)\otimes \Q_p/\Z_p \iri \dl{n} H_{fl}^1(X_{L_v},A[p^n])=: H_{fl}^1(X_{L_v},A[p^\infty])\ . \]

\begin{df}\label{SelDef}
The $p$ part of the Selmer group of $A$ over $L$ is defined as
\[ Sel_A(L)_p:=Ker\left\{ H^1_{fl}(X_L,A[p^\infty])\ri \prod_v H^1_{fl}(X_{L_v},A[p^\infty])/Im\,\kappa_{L_v}\right\} \]
where the map is the product of the natural restrictions at all primes $v$ of $L$. For an infinite algebraic extension
$\call/F$ we define, as usual, the Selmer group $Sel_A(\call)_p$ via the direct limit of the $Sel_A(L)_p$ for all
the finite subextensions $L$ of $\call$.
\end{df}

In this section we let $\calf_d/F$ be a $\Z_p^d$-extension ($d<\infty$) with Galois group $\Gm_d$ and associated Iwasawa algebra $\L_d\,$.
Our goal is to describe the structure of $Sel_A(\calf_d)_p$ (actually of its Pontrjagin dual) as a $\L_d$-module. The main step is a control
theorem proved in \cite{abl} for the case of elliptic curves and in \cite{tan} in general, which will enable us to prove that
$\cals(\calf_d):=Sel_A(\calf_d)_p^\vee$ is a finitely generated (in some cases torsion) $\L_d$-module. The proof of the control theorem
requires semi-stable reduction for $A$ at the places which ramify in $\calf_d/F$: this is not a restrictive hypothesis thanks to the
following (see \cite[Lemma 2.1]{OT})

\begin{lem}\label{OcTr}
Let $F'/F$ be a finite Galois extension. Let $\calf'_d:=\calf_d F' $ and $\L'_d:=\Z_p[[Gal(\calf'_d/F')]]$. Put $A'$ for the base change of
$A$ to $F'$. If $\cals':=Sel_{A'}(\calf'_d)_p^\vee$ is a finitely generated (torsion) $\L'_d$-module, then $\cals$ is a finitely generated
(torsion) $\L_d$-module.
\end{lem}

\begin{proof} From the natural embeddings $Sel_A(L)_p\iri H^1_{fl}(X_L, A[p^\infty])$ (any $L$) one gets a diagram between duals
\[ \xymatrix{ H^1_{fl}(X_{\calf'_d}, A[p^\infty])^\vee \ar@{->>}[rr] \ar[d] & & \cals' \ar[d]\\
H^1_{fl}(X_{\calf_d}, A[p^\infty])^\vee \ar@{->>}[rr] \ar@{->>}[d] & & \cals \ar@{->>}[d] \\
H^1(Gal(\calf'_d/\calf_d), A[p^\infty](\calf'_d))^\vee \ar@{->>}[rr] & & \cals/Im\,\cals' }\]
(where in the lower left corner one has the dual of a Galois cohomology group and the whole left side comes from the dual of the
Hochschild-Serre spectral sequence). Obviously $\calf'_d/\calf_d$ is finite (since $F'/F$ is) and $A[p^\infty](\calf'_d)$ is cofinitely
generated, hence $H^1(Gal(\calf'_d/\calf_d), A[p^\infty](\calf'_d))^\vee$ and $\cals/Im\,\cals'$ are finite as well. Therefore $\cals$
is a finitely generated (torsion) $\L'_d$-module and the lemma follows from the fact that $Gal(\calf'_d/F')$ is open in $\Gm_d\,$.
\end{proof}

\subsection{Elliptic curves}
Let $E/F$ be an elliptic curve, non-isotrivial (i.e., $j(E)\not\in \F_F\,$) and having good ordinary or split multiplicative reduction
at all the places which ramify in $\calf_d/F$ (assuming there is no ramified prime of supersingular reduction one just needs a finite
extension of $F$ to achieve this). We remark that for such curves $E[p^\infty](F^{sep})$ is finite (it
is an easy consequence of the fact that the $p^n$-torsion points for $n\gg 0$ generate inseparable extensions of $F$, see for example
\cite[Proposition 3.8]{blv}).\\
For any finite subextension $F\subseteq L \subseteq \calf_d$ we put $\Gm_L:= Gal(\calf_d/L)$ and consider the natural restriction map
\[ a_L\,:\,Sel_E(L)_p\longrightarrow Sel_E(\calf_d)_p^{\Gm_L}\ .\]

The following theorem summarizes results of \cite{abl} and \cite{tan}.\index{Control theorem}

\begin{thm}\label{CTEll}
In the above setting assume that $\calf_d/F$ is unramified outside a finite set of places of $F$ and that $E$ has good ordinary or split
multiplicative reduction at all ramified places. Then $Ker\,a_L$ is finite (of order bounded independently of $L$) and $Coker\,a_L$ is a
cofinitely generated $\Z_p$-module (of bounded corank if $d=1$). Moreover if all places of bad reduction for $E$ are unramified in $\calf_d/F$,
then $Coker\,a_L$ is finite as well (of bounded order if $d=1$).
\end{thm}

\begin{proof} Let $\calf_w$ be the completion of $\calf_d$ at $w$ and, to shorten notations, let
\[ \mathcal{G}(X_L):=Im\left\{H_{fl}^1(X_L,E[p^\infty])\longrightarrow
\prod_{v}\,H_{fl}^1(X_{L_v},E[p^\infty])/Im\,\kappa_{L_v}\right\} \]
(analogous definition for $\mathcal{G}(X_{\calf_d})\,$).\\
Consider the diagram
\[ \xymatrix{
Sel_E(L)_p \ar[d]^{a_L} \ar@{^{(}->}[r] & H_{fl}^1(X_L,E[p^\infty]) \ar[d]^{b_L} \ar@{->>}[r] &
\mathcal{G}(X_L)\ar[d]^{c_L}\\ Sel_E(\calf_d)_p^{\Gm_L} \ar@{^{(}->}[r] & H_{fl}^1(X_{\calf_d},E[p^\infty])^{\Gm_L} \ar[r] &
\mathcal{G}(X_{\calf_d})^{\Gm_L} \ .} \]
Since $X_{\calf_d}/X_L$ is a Galois covering the Hochschild-Serre spectral sequence (see \cite[III.2.21 a),b) and III.1.17 d)]{Mi1}) yields
\[ Ker\,b_L = H^1(\Gm_L,E[p^\infty](\calf_d)) \quad {\rm and} \quad
Coker\,b_L \subseteq H^2(\Gm_L,E[p^\infty](\calf_d)) \]
(where the $H^i$'s are Galois cohomology groups).
Since $E[p^\infty](\calf_d)$ is finite it is easy to see that
\[ |Ker\,b_L|\leqslant|E[p^\infty](\calf_d)|^d \quad {\rm and} \quad |Coker\,b_L|\leqslant|E[p^\infty](\calf_d)|^{\frac{d(d-1)}{2}} \]
(\cite[Lemma 4.1]{abl}). By the snake lemma the inequality on the left is enough to prove the first statement of the theorem, i.e.,
\[ |Ker\,a_L|\leqslant|Ker\,b_L|\leqslant|E[p^\infty](\calf_d)|^d \]
which is finite and bounded independently of $L$ (actually one can also use the upper bound $|E[p^\infty](F^{sep})|^d$ which makes it
independent of $\calf_d$ as well).

\noindent We are left with $Ker\,c_L\,$: for any place $w$ of $\calf_d$ dividing $v$ define
\[  d_w : H_{fl}^1(X_{L_v},E[p^\infty])/Im\,\kappa_{L_v} \longrightarrow H_{fl}^1(X_{\calf_w},E[p^\infty])/Im\,\kappa_{w}\ . \]
Then
\[ Ker\,c_L \iri \prod_{v} \bigcap_{w|v} Ker\,d_w \]
(note also that $Ker\,d_w$ really depends on $v$ and not on $w$). If $v$ totally splits in $\calf_d/L$ then $d_w$ is obviously an isomorphism.
Therefore from now on we study the $Ker\,d_w$'s only for primes which are not totally split. Moreover, because of the following diagram coming
from the Kummer exact sequence
\[ \xymatrix{
H_{fl}^1(X_{L_v},E[p^\infty])/Im\,\kappa_{L_v}\ar@{^{(}->}[r]\ar[d]^{d_w} & H_{fl}^1(X_{L_v},E)\ar[d]^{h_w} \\
H_{fl}^1(X_{\calf_w},E[p^\infty])/Im\,\kappa_{\calf_w}\ar@{^{(}->}[r] & H_{fl}^1(X_{\calf_w},E) \ , } \]
one has an injection
\[ Ker\,d_w \iri Ker\,h_w\simeq H^1(\Gm_{L_v},E(\calf_w)) \]
which allows us to focus on $H^1(\Gm_{L_v},E(\calf_w))$.

\subsubsection{Places of good reduction}\label{GoodRedEll}
If $v$ is unramified let $L_v^{unr}$ be the maximal unramified extension of $L_v\,$. Then using the inflation map and
\cite[Proposition I.3.8]{Mi2} one has
\[ H^1(\Gm_{L_v},E(\calf_w)) \iri H^1(Gal(L_v^{unr}/L_v),E(L_v^{unr}))=0\ .\]
Let $\wh{E}$ be the formal group associated with $E$ and, for any place $v$, let $E_v$ be the reduced curve. From the exact sequence
\[ \wh{E}(\ol_{\ov{L}_v})\iri E(\ov{L}_v)\sri E_v(\ov{\F}_v) \]
and the surjectivity of $E(L_v)\sri E_v(\F_v)$ (see \cite[Exercise I.4.13]{Mi1}), one gets
\[ H^1(\Gm_{L_v},\wh{E}(\ol_w))\iri H^1(\Gm_{L_v},E(\calf_w))\ri H^1(\Gm_{L_v},E_v(\F_w)) \ .\]
Using the Tate local duality (see \cite[Theorem III.7.8 and Appendix C]{Mi2}) and a careful study of the $p^n$-torsion points in inseparable
extensions of $L_v\,$, Tan proves that $H^1(\Gm_{L_v},\wh{E}(\ol_w))$ is isomorphic to the Pontrjagin dual of $E_v[p^\infty](\F_v)$
(see \cite[Theorem 2]{tan}). Hence
\[ |H^1(\Gm_{L_v},E(\calf_w))|\leqslant|H^1(\Gm_{L_v},E_v(\F_w))|\,|E_v[p^\infty](\F_v)| \ .\]
Finally let $L'_v$ be the maximal unramified extension of $L_v$ contained in $\calf_w$ (so that, in particular, $\F_w=\F_{L'_v}\,$)
and let $\Gm_{L'_v}:=Gal(\calf_w/L'_v)$ be the inertia subgroup of $\Gm_{L_v}$. The Hochschild-Serre sequence reads as
\begin{align}
H^1(\Gm_{L_v}/\Gm_{L'_v},E_v(\F_{L'_v})) \iri & H^1(\Gm_{L_v},E_v(\F_w))   \nonumber\\
 &\qquad\qquad \downarrow  \nonumber\\
 &  H^1(\Gm_{L'_v},E_v(\F_w))^{\Gm_{L_v}/\Gm_{L'_v}} \ri  H^2(\Gm_{L_v}/\Gm_{L'_v},E_v(\F_{L'_v}))\ .\nonumber \end{align}
Now $\Gm_{L_v}/\Gm_{L'_v}$ can be trivial, finite cyclic or $\Z_p$ and in any case Lang's Theorem yields
\[ H^i(\Gm_{L_v}/\Gm_{L'_v},E_v(\F_{L'_v}))=0\quad i=1,\ 2\ .\]
Therefore
\[ H^1(\Gm_{L_v},E_v(\F_w)) \simeq H^1(\Gm_{L'_v},E_v(\F_{L'_v}))^{Gal(L'_v/L_v)} \simeq H^1(\Gm_{L'_v},E_v(\F_v)) \]
and eventually
\begin{align} |H^1(\Gm_{L_v},E(\calf_w))| & \leqslant|H^1(\Gm_{L'_v},E_v(\F_v))|\,|E_v[p^\infty](\F_v)| \nonumber\\
 & \leqslant|E_v[p^\infty](\F_v)|^{d(L'_v)+1}\leqslant|E_v[p^\infty](\F_v)|^{d+1} \nonumber \end{align}
where $d(L'_v):=rank_{\Z_p}\Gm_{L'_v}\leqslant d$ and the middle bound (independent from $\calf_d\,$) comes from \cite[Lemma 4.1]{abl}.

We are left with the finitely many primes of bad reduction.

\subsubsection{Places of bad reduction}\label{BadRedEll}
By our hypothesis at these primes we have the Tate curve exact sequence
\[ q_{E,v}^\Z\iri \calf_w^*\sri E(\calf_w) \ .\]
For any subfield $K$ of $\calf_w/L_v$ one has a Galois equivariant isomorphism
\[ K^*/\ol_K^*q_{E,v}^\Z \simeq T_K  \]
(coming from $E_0(\calf_w)\iri E(\calf_w) \sri T_{\calf_w}\,$), where $T_K$ is a finite cyclic group of order $-ord_K(j(E))$ arising from
the group of connected components (see, for example, \cite[Lemma 4.9 and Remark 4.10]{abl}). Therefore
\[ H^1(\Gm_{L_v},E(\calf_w)) = \dl{K} H^1(\Gm_K ,E(K)) \iri \dl{K} H^1(\Gm_K ,T_K) \simeq \dl{K} (T_K)_p^{d(K)} \ , \]
where $(T_K)_p$ is the $p$-part of $T_K$ and $d(K)=\rank_{\Z_p} \Gm_K\,$.

If $v$ is unramified then all $T_K$'s are isomorphic to $T_{L_v}$ and $d(K)=d(L_v)=1$; hence
\[ |H^1(\Gm_{L_v} ,E(\calf_w))|=|(T_{L_v})_p|= |(T_{F_\nu})_p| \ ,\]
where $\nu$ is the prime of $F$ lying below $v$ (note that the bound is again independent of $\calf_d\,$).

If $v$ is ramified then taking Galois cohomology in the Tate curve exact sequence, one finds
\[ Ker\,h_w=H^1(\Gm_{L_v},E(\calf_w))\iri H^2(\Gm_{L_v},q_{E,v}^\Z) \]
where the injectivity comes from Hilbert 90.\\
Since $\Gm_{L_v}$ acts trivially on $q_{E,v}^\Z$, one finds that
\[ Ker\,h_w\iri H^2(\Gm_{L_v},q_{E,v}^\Z)\simeq H^2(\Gm_{L_v},\Z)\simeq (\Gm_{L_v}^{ab})^\vee\simeq(\Q_p/\Z_p)^{d(L_v)}\iri (\Q_p/\Z_p)^d\ , \]
i.e., $Ker\,h_w$ is a cofinitely generated $\Z_p$-module.

This completes the proof for the general case. If all ramified primes are of good reduction, then the $Ker\,d_w$'s are finite so
$Coker\,a_L$ is finite as well. In particular its order is bounded by
\[\prod_{v\ ram,\ good} |E_v[p^\infty](\F_{L_v})|^{d+1}\times \prod_{v\ inert,\ bad} p^{ord_p(ord_v(j(E)))} \times |Coker\,b_L|\ .\]
If $d=1$ the last term is trivial and, in a $\Z_p$-extension, the (finitely many) places which are ramified or inert of bad reduction
admit only a finite number of places above them. If $d\geqslant 2$ such bound is not independent of $L$ because the number of terms
in the products is unbounded. In the case of ramified primes of bad reduction the bound for the corank is similar.
\end{proof}

Both $Sel_E(\calf_d)_p$ and its Pontrjagin dual are modules over the ring $\L_d$ in a natural way. An easy consequence of the previous theorem
and of Nakayama's Lemma (see \cite{BH}) is the following (see for example \cite[Corollaries 4.8 and 4.13]{abl})

\begin{cor}\label{CTCorEll}
In the setting of the previous theorem, let $\cals(\calf_d)$ be the Pontrjagin dual of $Sel_E(\calf_d)_p\,$.
Then $\cals(\calf_d)$ is a finitely generated $\L_d$-module.
Moreover if all ramified primes are of good reduction and $Sel_E(F)_p$ is finite, then $\cals(\calf_d)$ is $\L_d$-torsion.
\end{cor}

\begin{rems}
\ \begin{itemize}
\item[{\bf 1.}] We recall that, thanks to Lemma \ref{OcTr}, the last corollary holds when there are no ramified primes of supersingular
reduction for $E$ (when such a prime is present the finitely generated statement does not hold anymore, see \cite[Theorem 3.10]{tan2}).
\item[{\bf 2.}] The ramified primes of split multiplicative reduction are the only obstacle to the finiteness of $Coker\,a_L$ and this
somehow reflects the number field situation as described in \cite[section II.6]{MTT}, where the authors defined an extended Mordell-Weil
group whose rank is $rank\,E(F)+N$ (where $N$ is the number of primes of split multiplicative reduction and dividing $p$, i.e., totally
ramified in the cyclotomic $\Z_p$-extension they work with) to deal with the phenomenon of exceptional zeroes.
\item[{\bf 3.}] A different way of having finite kernels and cokernels (and then, at least in some cases, torsion modules $\cals(\calf_d)\,$)
consists in a modified version of the Selmer groups. Examples with trivial or no conditions at all at the ramified primes of bad reduction
are described in \cite[Theorem 4.12]{abl}.
\item[{\bf 4.}] The available constructions of a $p$-adic $L$-function associated to $\Z_p^\N$-extensions require the presence of a totally
ramified prime $\pr$ of split multiplicative reduction for $E$. Thus the theorem applies to that setting but, unfortunately, it only provides
finitely generated $\L_d$-modules $\cals(\calf_d)$.
\item[{\bf 5.}] The paper \cite{LLTT1} describes an example of an elliptic curve $E/F$ and a $\Z_p$-extension $\calf_1$ such
that $\cals(\calf_1)$ is a non-torsion $\L_1$-module (the last section of \cite{LLTT1} verifies the vanishing of the $p$-adic
$L$-function attached to these $E$ and $\calf$, in accordance with the Iwasawa Main Conjecture).
\end{itemize}
\end{rems}

\subsection{Higher dimensional abelian varieties}
We go back to the general case of an abelian variety $A/F$.
For any finite subextension $L/F$ of $\calf_d$ we put $\Gm_L:= Gal(\calf_d/L)$ and consider the natural restriction map
\[ a_L\,:\,Sel_A(L)_p\longrightarrow Sel_A(\calf_d)_p^{\Gm_L}\ .\]

The following theorem summarizes results of \cite{tan}.\index{Control theorem}

\begin{thm}\label{CTAbVar}
In the above setting assume that $\calf_d/F$ is unramified outside a finite set of places of $F$ and that $A$ has good ordinary or
split multiplicative reduction at all ramified places. Then $Ker\,a_L$ is finite (of bounded order if $d=1$) and $Coker\,a_L$ is a
cofinitely generated $\Z_p$-module. Moreover if all places of bad reduction for $A$ are unramified in $\calf_d/F$, then $Coker\,a_L$ is
finite as well (of bounded order if $d=1$).
\end{thm}

\begin{proof} We use the same notations and diagrams as in Theorem \ref{CTEll}, substituting the abelian variety $A$ for the elliptic curve $E$.\\
The Hochschild-Serre spectral sequence yields
\[ Ker\,b_L = H^1(\Gm_L,A[p^\infty](\calf_d)) \quad {\rm and} \quad
Coker\,b_L \subseteq H^2(\Gm_L,A[p^\infty](\calf_d)) \ .\]

Let $L_0\subseteq \calf_d$ be the extension generated by $A[p^\infty](\calf_d)$. The extension $L_0/L$ is everywhere unramified (for the places
of good reduction see \cite[Lemma 2.5.1 (b)]{tan}, for the other places note that the $p^n$-torsion points come from the $p^n$-th roots of the
periods provided by the Mumford parametrization so they generate an inseparable extension while $\calf_d/F$ is separable): hence
$Gal(L_0/L)\simeq \Delta \times \Z_p^e$ where $\Delta$ is finite and $e=0$ or 1. Let $\gamma$ be a topological generator of $\Z_p^e$ in
$Gal(L_0/L)$ (if $e=0$ then $\gamma=1$) and let $L_1$ be its fixed field. Then $A[p^\infty](\calf_d)^{\overline{<\gamma >}}=A[p^\infty](L_1)$
is finite and we can apply \cite[Lemma 3.4]{abl2} (with $b$ the maximum between $|A[p^\infty](L_1)|$ and
$|A[p^\infty](\calf_d)/(A[p^\infty](\calf_d))_{div}|$) to get
\[ |Ker\,b_L|\leqslant b^d \quad {\rm and} \quad |Coker\,b_L|\leqslant b^{\frac{d(d-1)}{2}}\ . \]
By the snake lemma the inequality on the left is enough to prove that $Ker\,a_L$ is finite (for the bounded order in the case $d=1$
see \cite[Corollary 3.2.4]{tan}).

\noindent The bounds for the $Ker\,d_w$'s are a direct generalization of the ones provided for the case of the elliptic curve so we give
just a few details. Recall the embedding
\[ Ker\,d_w \iri Ker\,h_w\simeq H^1(\Gm_{L_v},E(\calf_w)) \ .\]

\subsubsection{Places of good reduction} If $v$ is unramified then
\[ H^1(\Gm_{L_v},A(\calf_w)) \iri H^1(Gal(L_v^{unr}/L_v),A(L_v^{unr}))=0\ .\]

If $v$ is ramified one has an exact sequence (as above)
\[ H^1(\Gm_{L_v},\wh{A}(\ol_{\calf_w}))\iri H^1(\Gm_{L_v},A(\calf_w))\ri H^1(\Gm_{L_v},A_v(\F_{\calf_w})) \ .\]
By \cite[Theorem 2]{tan}
\[ H^1(\Gm_{L_v},\wh{A}(\ol_{\calf_w}))\simeq B_v[p^\infty](\F_{L_v}) \]
(where $B$ is the dual variety of $A$) and the last group has the same order of $A_v[p^\infty](\F_{L_v})$. Using Lang's theorem as in
\ref{GoodRedEll}, one finds
\[ |H^1(\Gm_{L_v},A(\calf_w))|\leqslant|A_v[p^\infty](\F_{L_v})|^{d+1} \ .\]

\subsubsection{Places of bad reduction} If $v$ is unramified let $\pi_{0,v}(A)$ be the group of connected components of the N\'eron model of
$A$ at $v$. Then, again by \cite[Proposition I.3.8]{Mi2},
\[ H^1(\Gm_{L_v},A(\calf_w)) \iri H^1(Gal(L_v^{unr}/L_v),A(L_v^{unr}))\simeq H^1(Gal(L_v^{unr}/L_v),\pi_{0,v}(A)) \]
and the last group has order bounded by $|\pi_{0,v}(A)^{Gal(L_v^{unr}/L_v)}|$.

If $v$ is ramified one just uses Mumford's parametrization with a period lattice $\Omega_v\subset L_v\times\dots\times L_v$ (genus $A$ times)
to prove that $H^1(\Gm_{L_v},A(\calf_w))$ is cofinitely generated as in \ref{BadRedEll}.
\end{proof}

We end this section with the analogue of Corollary \ref{CTCorEll}.

\begin{cor}
In the setting of the previous theorem, let $\cals(\calf_d)$ be the Pontrjagin dual of $Sel_A(\calf_d)_p\,$.
Then $\cals(\calf_d)$ is a finitely generated $\L_d$-module.
Moreover if all ramified primes are of good reduction and $Sel_A(F)_p$ is finite, then $\cals(\calf_d)$ is $\L_d$-torsion.
\end{cor}

\begin{rem}
In \cite[Theorem 1.7]{OT}, by means of crystalline and syntomic cohomology, Ochiai and Trihan prove a stronger result. Indeed they can show that
the dual of the Selmer group is always torsion, with no restriction on the abelian variety $A/F$, but only in the case of the arithmetic
extension $\calf^{arit}/F$, which lies outside the scope of the present paper. Moreover in the case of a (not necessarily commutative)
pro-$p$-extension containing $\calf^{arit}\,$, they prove that the dual of the Selmer group is finitely generated (for a precise statement,
see \cite{OT}, in particular Theorem 1.9)\,\footnote{A generalization of the torsion statement for $\Z_p^d$-extensions containing $\calf^{arit}$
can be found in \cite{tan2} and \cite{BBL2} (where one also finds an approach to the results of Section \ref{LamMod} in terms of
characteristic ideals).}.
\end{rem}

\section{$\Lambda$-modules and Fitting ideals}\label{LamMod}
We need a few more notations.\\
For any $\Z_p^d$-extension $\calf_d\,$, let $\Gamma(\calf_d):=Gal(\calf_d/F)$ and $\Lambda(\calf_d):=\Z_p[[\Gamma(\calf_d)]]$
(the Iwasawa algebra) with augmentation ideal $I^{\calf_d}$ (or simply $\Gamma_d\,$, $\Lambda_d$ and $I_d$ if the extension $\calf_d$ is
clearly fixed).\\
For any $d>e$ and any $\Z_p^{d-e}$-extension $\calf_d/\calf_e\,$, we put $\Gamma(\calf_d/\calf_e):=Gal(\calf_d/\calf_e)$,
$\Lambda(\calf_d/\calf_e):=\Z_p[[\Gm(\calf_d/\calf_e)]]$ and $I^{\calf_d}_{\calf_e}$ as the augmentation ideal of $\Lambda(\calf_d/\calf_e)$,
i.e., the kernel of the canonical projection $\pi^{\calf_d}_{\calf_e}:\Lambda(\calf_d)\rightarrow \Lambda(\calf_e)$ (whenever possible all these
will be abbreviated to $\Gamma^d_e\,$, $\Lambda^d_e\,$, $I^d_e$ and $\pi^d_e$ respectively).\\
Recall that $\L(\calf_d)$ is (noncanonically) isomorphic to $\Z_p[[T_1,\dots ,T_d\,]]$. A finitely generated torsion $\L(\calf_d)$-module
is said to be {\em pseudo-null} if its annihilator ideal has height at least 2. If $M$ is a finitely generated torsion $\L(\calf_d)$-module, then
there is a pseudo-isomorphism (i.e., a morphism with pseudo-null kernel and cokernel)
\[ M\sim_{\L(\calf_d)} \bigoplus_{i=1}^n \L(\calf_d)/(g_i^{e_i})  \ ,\]
where the $g_i$'s are irreducible elements of $\L(\calf_d)$ (determined up to an element of $\L(\calf_d)^*\,$)
and $n$ and the $e_i$'s are uniquely determined by $M$ (see e.g. \cite[VII.4.4 Theorem 5]{Bo}).

\begin{df}\label{defCharid}\index{Characteristic ideal}
In the above setting the {\em characteristic ideal} of $M$ is
\[ Ch_{\L(\calf_d)}(M):=\left\{ \begin{array}{ll}
0 & {\rm if\ }M\ {\rm is\ not\ torsion} \\
\displaystyle{\left(\,\prod_{i=1}^n g_i^{e_i}\right) }& {\rm otherwise}
\end{array} \right. \ .\]
\end{df}

Let $Z$ be a finitely generated $\L(\calf_d)$-module and let
\[ \L(\calf_d)^a {\buildrel\varphi\over\longrightarrow} \L(\calf_d)^b \sri Z \]
be a presentation where the map $\varphi$ can be represented by a $b\times a$ matrix $\Phi$ with entries in $\L(\calf_d)\,$.

\begin{df}\label{defFittid}\index{Fitting ideal}
In the above setting the {\em Fitting ideal} of $Z$ is
\[ Fitt_{\L(\calf_d)}(Z):=\left\{ \begin{array}{ll}
0 & {\rm if}\ a<b \\
{\rm the\ ideal\ generated\ by\ all\ the}& \\
{\rm determinants\ of\ the\ } b\times b & {\rm if}\ a\geqslant b \\
{\rm minors\ of\ the\ matrix\ } \Phi &
\end{array} \right. \ .\]
\end{df}

Let $\calf/F$ be a $\Z_p^\N$-extension with Galois group $\Gm$. Our goal is to define an ideal in $\L:=\Z_p[[\Gm]]$ associated with $\cals$,
the Pontrjagin dual of $Sel_A(\calf)_p\,$. For this we consider all the $\Z_p^d$-extensions $\calf_d/F$ ($d\in \N$) contained in $\calf$ (which
we call $\Z_p$-finite extensions). Then $\calf=\cup \calf_d$ and $\L=\liminv \L(\calf_d):=\liminv \Z_p[[Gal(\calf_d/F)]]$. The classical
characteristic ideal does not behave well (in general) with respect to inverse limits (because the inverse limit of pseudo-null modules is
not necessarily pseudo-null). For the Fitting ideal, using the basic properties described in the Appendix of \cite{MW}, we have the following

\begin{lem}\label{PrLimFitt}
Let $\calf_d\subset \calf_e$ be an inclusion of multiple $\Z_p$-extensions, $e > d$. Assume that $A[p^\infty](\calf)=0$ or that
$Fitt_{\L(\calf_d)}(\cals(\calf_d))$ is principal. Then
\[ \pi^{\calf_e}_{\calf_d}(Fitt_{\L(\calf_e)}(\cals(\calf_e)))\subseteq Fitt_{\L(\calf_d)}(\cals(\calf_d))\ . \]
\end{lem}

\begin{proof} Consider the natural map $a^e_d : Sel_A(\calf_d)_p\ri Sel_A(\calf_e)_p^{\Gamma^e_d}$ and dualize to get
\[ \cals(\calf_e)/I^e_d\cals(\calf_e) \ri \cals(\calf_d) \sri (Ker\,a^e_d)^\vee \]
where (as in Theorem \ref{CTAbVar})
\[ Ker\,a^e_d \iri H^1(\Gamma^e_d,A[p^\infty](\calf_e)) \]
is finite. If $A[p^\infty](\calf)=0$ then $(Ker\,a^e_d)^\vee=0$ and
\[ \pi^e_d(Fitt_{\L_e}(\cals(\calf_e)))= Fitt_{\L_d}(\cals(\calf_e)/I^e_d\cals(\calf_e))\subseteq
Fitt_{\L_d}(\cals(\calf_d)) \ .\]
If $(Ker\,a^e_d)^\vee\neq 0$ one has
\[ Fitt_{\L_d}(\cals(\calf_e)/I^e_d\cals(\calf_e))Fitt_{\L_d}((Ker\,a^e_d)^\vee) \subseteq
Fitt_{\L_d}(\cals(\calf_d)) \ .\]
The Fitting ideal of a finitely generated torsion module contains a power of its annihilator, so let $\sigma_1\,,\sigma_2$ be two relatively
prime elements of $Fitt_{\L_d}((Ker\,a^e_d)^\vee)$ and $\theta_d$ a generator of $Fitt_{\L_d}(\cals(\calf_d))$. Then $\theta_d$ divides
$\sigma_1\alpha$ and $\sigma_2\alpha$ for any $\alpha\in Fitt_{\L_d}(\cals(\calf_e)/I^e_d\cals(\calf_e))$ (it holds, in the obvious sense,
even for $\theta_d=0$). Hence
\[  \pi^e_d(Fitt_{\L_e}(\cals(\calf_e)))=Fitt_{\L_d}(\cals(\calf_e)/I^e_d\cals(\calf_e)) \subseteq Fitt_{\L_d}(\cals(\calf_d)) \ . \]
\end{proof}

\begin{rem}
In the case $A=E$ an elliptic curve, the hypothesis $E[p^\infty](\calf)=0$ is satisfied if $j(E)\not\in (F^*)^p$, i.e., when the curve is
admissible (in the sense of \cite{blv}); otherwise $j(E)\in (F^*)^{p^n}-(F^*)^{p^{n+1}}$ and one can work over the field $F^{p^n}$. The
other hypothesis is satisfied in general by elementary $\L(\calf_d)$-modules or by modules having a presentation with the same
number of generators and relations.
\end{rem}

Let $\pi_{\calf_d}$ be the canonical projection from $\L$ to $\L(\calf_d)$ with kernel $I_{\calf_d}\,$. Then the previous lemma shows that,
as $\calf_d$ varies, the $(\pi_{\calf_d})^{-1}(Fitt_{\L(\calf_d)}(\cals(\calf_d)))$ form an inverse system of ideals in $\L$.

\begin{df}
Assume that $A[p^\infty](\calf)=0$ or that $Fitt_{\L(\calf_d)}(\cals(\calf_d))$ is principal for any $\calf_d\,$. Define
\[ \wt{Fitt}_\L(\cals(\calf)):=\il{\calf_d} (\pi_{\calf_d})^{-1}(Fitt_{\L(\calf_d)}(\cals(\calf_d))) \]
to be the \emph{pro-Fitting ideal} of $\cals(\calf)$ (the Pontrjagin dual of $Sel_E(\calf)_p\,$).
\end{df}

\begin{prop}\label{AugFitt}
Assume that $A[p^\infty](\calf)=0$ or that $Fitt_{\L(\calf_d)}(\cals(\calf_d))$ is principal for any $\calf_d\,$. If
$corank_{\Z_p}Sel_A(\calf_1)_p^{\Gamma(\calf_1)}\geqslant 1$ for any $\Z_p$-extension $\calf_1/F$ contained in $\calf$, then
$\wt{Fitt}_\L(\cals(\calf))\subset I$ (where $I$ is the augmentation ideal of $\L$).
\end{prop}

\begin{proof} Recall that $I^{\calf_d}$ is the augmentation ideal of $\L(\calf_d)$, that is, the kernel of
$\pi^{\calf_d}\colon\L(\calf_d)\ri \Z_p\,$.
By hypothesis $Fitt_{\Z_p}((Sel_A(\calf_1)_p^{\Gamma(\calf_1)})^\vee)=0$. Thus, since
$\Z_p=\L(\calf_1)/I^{\calf_1}$ and $(Sel_A(\calf_1)_p^{\Gamma(\calf_1)})^\vee=\cals(\calf_1)/I^{\calf_1}\cals(\calf_1)$,
\[ 0=Fitt_{\Z_p}((Sel_A(\calf_1)_p^{\Gamma(\calf_1)})^\vee)=\pi^{\calf_1}(Fitt_{\L(\calf_1)}(\cals(\calf_1))) \ ,\]
i.e., $Fitt_{\L(\calf_1)}(\cals(\calf_1))\subset Ker\,\pi^{\calf_1}=I^{\calf_1}\,$. \\
For any $\Z_p^d$-extension $\calf_d$ take a $\Z_p$-extension $\calf_1$ contained in $\calf_d\,$. Then, by Lemma \ref{PrLimFitt},
\[ \pi^{\calf_d}_{\calf_1}(Fitt_{\L(\calf_d)}(\cals(\calf_d)))\subseteq Fitt_{\L(\calf_1)}(\cals(\calf_1)) \subset I^{\calf_1} \ .\]
Note that $\pi^{\calf_d}=\pi^{\calf_1}\circ \pi^{\calf_d}_{\calf_1}\,$.
Therefore
\[ Fitt_{\L(\calf_d)}(\cals(\calf_d))\subset I^{\calf_d} \iff \pi^{\calf_d}_{\calf_1}(Fitt_{\L(\calf_d)}(\cals(\calf_d)))\subset I^{\calf_1} \ ,\]
i.e., $Fitt_{\L(\calf_d)}(\cals(\calf_d))\subset I^{\calf_d}$ for any $\Z_p$-finite extension $\calf_d\,$.
Finally
\[ \wt{Fitt}_\L(\cals(\calf)):= \bigcap_{\calf_d} (\pi_{\calf_d})^{-1}(Fitt_{\L(\calf_d)}(\cals(\calf_d))) \subset
\bigcap_{\calf_d} (\pi_{\calf_d})^{-1}(I^{\calf_d}) \subset I \ .\]
\end{proof}

\begin{rem}
From the exact sequence
\[ Ker\,a_{\calf_1} \iri Sel_A(F)_p {\buildrel{\ \ a_{\calf_1}}\over-\!\!\!\!\!\longrightarrow} Sel_A(\calf_1)_p^{\Gamma(\calf_1)}
\sri Coker\,a_{\calf_1} \]
and the fact that $Ker\,a_{\calf_1}$ is finite one immediately finds out that the hypothesis on $corank_{\Z_p}Sel_A(\calf_1)_p^{\Gamma(\calf_1)}$
is satisfied if $rank_\Z A(F)\geqslant 1$ or $corank_{\Z_p}Coker\,a_{\calf_1} \geqslant 1$. As already noted, when there is a totally ramified prime of
split multiplicative reduction, the second option is very likely to happen. In the number field case, when $\calf$ is the cyclotomic
$\Z_p$-extension and, in some cases, $Sel_E(\calf)_p^\vee$ is known to be a torsion module, this is equivalent to saying that $T$ divides a
generator of the characteristic ideal of $Sel_E(\calf)_p^\vee$ (i.e., there is an exceptional zero). Note that all the available constructions
of $p$-adic $L$-function for our setting require a ramified place of split multiplicative reduction and they are all known to belong to $I$.
\end{rem}

\section{Modular abelian varieties of $GL_2$-type} \label{s:modGL2}

The previous sections show how to define the algebraic ($p$-adic) $L$-function associated with $\calf/F$ and an abelian variety $A/F$ under
quite general conditions. On the analytic side there is, of course, the complex Hasse-Weil $L$-function $L(A/F,s)$, so the problem becomes to
relate it to some element in an Iwasawa algebra. In this section we will sketch how this can be done at least in some cases; in order to keep
the paper to a reasonable length, the treatment here will be very brief.\bigskip

\index{Abelian variety of $GL_2$-type} We say that the abelian variety $A/F$ is of $GL_2$-type if there is a number field $K$ such that
$[K:\Q]=\dim A$ and $K$ embeds into $\End_F(A)\otimes\Q$. In particular, this implies that for any $l\neq p$ the Tate module $T_lA$ yields
a representation of $G_F$ in $GL_2(K\otimes\Q_l$). The analogous definition for $A/\Q$ can be found in \cite{ribet}, where it is proved that
Serre's conjecture implies that every simple abelian variety of $GL_2$-type is isogenous to a simple factor of a modular Jacobian. We are going
to see that a similar result holds at least partially in our function field setting.

\subsection{Automorphic forms} Let $\ad_F$ denote the ring of adeles of $F$. By automorphic form for $GL_2$ we shall mean a function
$f\colon GL_2(\ad_F)\longrightarrow\C$ which factors through $GL_2(F)\backslash GL_2(\ad_F)/\calk$, where  $\calk$ is some open compact
subgroup of $GL_2(\ad_F)$; furthermore, $f$ is cuspidal if it satisfies some additional technical condition (essentially, the annihilation
of some Fourier coefficients). A classical procedure associates with such an $f$ a Dirichlet sum $L(f,s)$: see e.g.
\cite[Chapters II and III]{weildir}.

The $\C$-vector spaces of automorphic and cuspidal forms provide representations of $GL_2(\ad_F)$. Besides, they have a natural $\Q$-structure:
in particular, the decomposition of the space of cuspidal forms in irreducible representations of $GL_2(\ad_F)$ holds over $\overline\Q$ (and
hence over any algebraically closed field of characteristic zero); see e.g.\;the discussion in \cite[page 218]{vdprev}. We also recall that
every irreducible automorphic representation $\pi$ of $GL_2(\ad_F)$ is a restricted tensor product $\otimes_v'\pi_v$, $v$ varying over the places
of $F$: the $\pi_v$'s are representations of $GL_2(F_v)$ and they are called local factors of $\pi$.

Let $W_F$ denote the Weil group of $F$: it is the subgroup of $G_F$ consisting of elements whose restriction to $\overline\F_q$ is an integer
power of the Frobenius. By a fundamental result of Jacquet and Langlands \cite[Theorem 12.2]{jaclan}, a two-dimensional representation of
$W_F$ corresponds to a cuspidal representation if the associated $L$-function and its twists by characters of $W_F$ are entire functions bounded
in vertical strips (see also \cite{weildir}).

Let $A/F$ be an abelian variety of $GL_2$-type. Recall that $L(A/F,s)$ is the $L$-function associated with the compatible system of
$l$-adic representations of $G_F$ arising from the Tate modules $T_lA$, as $l$ varies among primes different from $p$. Theorems of Grothendieck
and Deligne show that under certain assumptions $L(A/F,s)$ and all its twists are polynomials in $q^{-s}$ satisfying the conditions of
\cite[Theorem 12.2]{jaclan} (see \cite[\S9]{delcost} for precise statements). In particular all elliptic curves are obviously of $GL_2$-type and
one finds that $L(A/F,s)=L(f,s)$ for some cusp form $f$ when $A$ is a non-isotrivial elliptic curve.

\subsection{Drinfeld modular curves} From now on we fix a place $\infty$.

The main source for this section is Drinfeld's original paper \cite{drinf}. Here we just recall that for any divisor $\gotn$ of $F$ with
support disjoint from $\infty$ there exists a projective curve $M(\gotn)$ (the Drinfeld modular curve) and that these curves form a projective
system. Hence one can consider the Galois representation
$$\underline H:=\limdir H^1_{et}(M(\gotn)\times F^{sep},\overline\Q_l).$$
Besides, the moduli interpretation of the curves $M(\gotn)$ allows to define an action of $GL_2(\ad_f)$ on $\underline H$ (where $\ad_f$ denotes
the adeles of $F$ without the component at $\infty$). Let $\Pi_\infty$ be the set of those cuspidal representations having the special
representation of $GL_2(F_\infty)$ (i.e., the Steinberg representation) as local factor at $\infty$. \index{Drinfeld's reciprocity law}
Drinfeld's reciprocity law \cite[Theorem 2]{drinf} (which realizes part of the Langlands correspondence for $GL_2$ over $F$) attaches to
any $\pi\in\Pi_\infty$ a compatible system of two-dimensional Galois representations $\sigma(\pi)_l\colon G_F\longrightarrow GL_2(\overline\Q_l)$
by establishing an isomorphism of $GL_2(\ad_f)\times G_F$-modules
\begin{equation} \label{e:drinfrec} \underline H\simeq \bigoplus_{\pi\in\Pi_{\infty}}(\otimes_{v\neq\infty}'\pi_v)\otimes
\sigma(\pi)_l\ . \end{equation}
As $\sigma(\pi)_l$ one obtains all $l$-adic representations of $G_F$ satisfying certain properties: for a precise list
(and a thorough introduction to all this subject) see \cite{vdprev}. Here we just remark the following requirement:
the restriction of  $\sigma(\pi)_l$ to $G_{F_{\infty}}$ has to be the special $l$-adic Galois representation $sp_{\infty}$. For example,
the representation originated from the Tate module $T_lE$ of an elliptic curve $E/F$ satisfies this condition if and only if $E$ has split
multiplicative reduction at $\infty$.

The Galois representations appearing in \eqref{e:drinfrec} are exactly those arising from the Tate module of the Jacobian of some
$M(\gotn)$. \index{Modular abelian varieties} We call modular those abelian varieties isogenous to some factor of $Jac(M(\gotn))$. Hence we see
that a necessary condition for abelian varieties of $GL_2$-type to be modular is that their reduction at $\infty$ is a split torus.\bigskip

The paper \cite{gekrev} provides a careful construction of Jacobians of Drinfeld modular curves by means of rigid analytic geometry.

\subsection{The $p$-adic $L$-functions} For any ring $R$ let $Meas(\field{P}^1(F_v),R)$ denote the $R$-valued measures on the topological
space $\field{P}^1(F_v)$ (that is, finitely additive functions on compact open subsets of $\field{P}^1(F_v)$) and $Meas_0(\field{P}^1(F_v),R)$
the subset of measures of total mass 0. A key ingredient in the proof of \eqref{e:drinfrec} is the identification of the space of $R$-valued
cusp forms with direct sums of certain subspaces of $Meas_0(\field{P}^1(F_\infty),R)$ (for more precise statements, see \cite[\S2]{vdprev} and
\cite[\S4]{gekrev}). Therefore we can associate with any modular abelian variety $A$ some measure $\mu_A$ on $\field{P}^1(F_\infty)$; this fact
can be exploited to construct elements (our $p$-adic $L$-functions) in Iwasawa algebras in the following way.

Let $K$ be a quadratic algebra over $F$: an embedding of $K$ into the $F$-algebra of $2\times2$ matrices $M_2(F)$ gives rise to an action of the
group $G:=(K\otimes F_\infty)^*/F_\infty^*$ on the $PGL_2(F_\infty)$-homogeneous space $\field{P}^1(F_\infty)$. Class field theory permits to
relate $G$ to a subgroup $\Gamma$ of $\tilde\Gm=Gal(\calf/F)$, where $\calf$ is a certain extension of $F$ (depending on $K$) ramified only
above $\infty$. Then the pull-back of $\mu_A$ to $G$ yields a measure on $\Gm$; this is enough because $Meas(\Gm,R)$ is canonically identified
with $R\otimes\L$ (and $Meas_0(\Gm,R)$ with the augmentation ideal). Various instances of the construction just sketched are explained in \cite{l}
for the case when $A$ is an elliptic curve: here one can take $R=\Z$. Similar ideas were used in P\'al's thesis \cite{pal}, where there is also
an interpolation formula relating the element in $\Z[[\Gamma]]$ so obtained to special values of the complex $L$-function. One should also
mention \cite{pal2} for another construction of $p$-adic $L$-function, providing an interpolation property for one of the cases studied in
\cite{l}. Notice that in all this cases the $p$-adic $L$-function is, more or less tautologically, in the augmentation ideal.

A different approach had been previously suggested by Tan \cite{tanmod}: starting with cuspidal automorphic forms, he defines elements in
certain group algebras and proves an interpolation formula \cite[Proposition \,2]{tanmod}. Furthermore, if the cusp form is ``well-behaved''
his modular elements form a projective system and originate an element in an Iwasawa algebra of the kind considered in the present paper:
in particular, this holds for non-isotrivial elliptic curves having split multiplicative reduction. In the case of an elliptic curve over
$\F_q(T)$ Teitelbaum \cite{teit} re-expressed Tan's work in terms of modular symbols (along the lines of \cite{MTT}); in \cite{hl} it is shown
how this last method can be related to the ``quadratic algebra'' techniques sketched above.

A unified treatment of all of this will appear in \cite{abl3}.\bigskip

Thus for a modular abelian variety $A/F$ we can define both a Fitting ideal and a $p$-adic $L$-function: it is natural to expect that an
Iwasawa Main Conjecture should hold, i.e., that the Fitting ideal should be generated by the $p$-adic $L$-function.

\begin{rem} In the cases considered in this paper (involving a modular abelian variety and a geometric extension of the function field)
the Iwasawa Main Conjecture is still wide open. However, recently there has been some interesting progress in two related settings.

First, one can take $A$ to be an isotrivial abelian variety (notice that \cite[page 308]{tanmod} defines modular elements also for an
isotrivial elliptic curve). Thanks to an observation of Ki-Seng Tan, the Main Conjecture in this setting can be reduced to the one
for class groups, which is already known to hold (as it will be explained in the next section). On this basis,
the Iwasawa Main Conjecture for constant ordinary abelian varieties is proved in \cite{LLTT2} when $\calf$ is a $\Z_p^d$-extension.

Second, one can take as $\calf$ the maximal arithmetic pro-$p$-extension of $F$, i.e., $\calf=\calf^{arit}=F\F_F^{(p)}$, where $\F_F^{(p)}$
is the subfield of $\overline\F_F$ defined by $Gal(\F_F^{(p)}/\F_F)\simeq\Z_p$ (note that this is the setting of \cite[Theorem 1.7]{OT}).
In this case Trihan has obtained a proof of the Iwasawa Main Conjecture, by techniques of syntomic cohomology. No
assumption on the abelian variety $A/F$ is needed: the relevant $p$-adic $L$-function is defined by means of cohomology and
it interpolates the Hasse-Weil $L$-function (see \cite{LLTT3}).
\end{rem}

\section{Class groups} \label{s:classgroups}

For any finite extension $L/F$, $\cala(L)$ will denote the $p$-part of the group of degree zero divisor classes of $L$; for any $\calf'$
intermediate between $F$ and $\calf$, we put $\cala(\calf'):=\liminv\cala(L)$ as $L$ runs among finite subextensions of $\calf'/F$ (the limit
being taken with respect to norm maps). The study of similar objects and their relations with zeta functions is an old subject and was the
starting point for Iwasawa himself (see \cite{thak} for a quick summary). The goal of this section is to say something on what is known about
Iwasawa Main Conjectures for class groups in our setting.

\subsection{Crew's work} A version of the Iwasawa Main Conjecture over global function fields was proved by R. Crew in \cite{crew2}.
His main tools are geometric: so he considers an affine curve $X$ over a finite field of characteristic $p$ (in the language of the
present paper, $F$ is the function field of $X$) and a $p$-adic character of $\pi_1(X)$, that is, a continuous homomorphism
$\rho\colon\pi_1(X)\longrightarrow R^*$, where $R$ is a complete local noetherian ring of mixed characteristic, with maximal ideal
$\mathfrak m$ (notice that the Iwasawa algebras $\L_d$ introduced in section \ref{ss:selmer} above are rings of this kind).
To such a $\rho$ are attached $H(\rho,x)\in R[x]$ (the characteristic polynomial of the cohomology of a certain \'etale sheaf - see
\cite{crew} for more explanation) and the $L$-function $L(\rho,x)\in R[[x]]$. The main theorem of \cite{crew2} proves, by means of
\'etale and crystalline cohomology, that the ratio $L(\rho,x)/H(\rho,x)$ is a unit in the $\mathfrak m$-adic completion of $R[x]$.
An account of the geometric significance of this result (together with some of the necessary background) is provided by Crew himself
in \cite{crew}; in \cite[\S3]{crew2} he shows the following application to Iwasawa theory. Letting (in our notations) $R$ be the
Iwasawa algebra $\L(\calf_d)$, the special value $L(\rho,1)$ can be seen to be a Stickelberger element (the definition will be recalled
in section \ref{s:stickelberger} below). As for  $H(\rho,1)$, \cite[Proposition 3.1]{crew} implies that it generates the characteristic
ideal of the torsion $\L_d$-module $\liminv\cala(L)^\vee$, $L$ varying among finite subextensions of $\calf_d/F$ \footnote{Note that
in \cite{crew} our $\cala(L)$'s appear as Picard groups, so the natural functoriality yields $\cala(L)\rightarrow\cala(L')$
if $L\subset L'$ - that is, arrows are opposite to the ones we consider in this paper: hence Crew takes Pontrjagin duals and we don't.}.
The Iwasawa Main Conjecture follows.

Crew's cohomological techniques are quite sophisticated. A more elementary approach was suggested by Kueh, Lai and Tan in \cite{klt}
(and refined, with Burns's contribution and different cohomological tools, in \cite{blt09}). In the next two sections we will give a brief
account of this approach (and its consequences) in a particularly simple setting, related to Drinfeld-Hayes cyclotomic extensions (which will
be the main topic of section \ref{s:carlitz}).

\subsection{Characteristic ideals for class groups} \label{s:charideal}
In this section\index{Characteristic ideal} (which somehow parallels section \ref{LamMod}) we describe an algebraic object which can be associated
to the inverse limit of class groups in a $\Z_p^\N$-extension $\calf$ of a global function field $F$. Since our first goal is to use this
``algebraic $L$-function'' for the cyclotomic extension which will appear in section \ref{CycNot}, we make the following simplifying assumption.

\begin{ass} \label{assumpt} There is only one ramified prime in $\calf/F$ (call it $\pr$) and it is totally ramified (in
particular this implies that $\calf$ is disjoint from $\calf^{arit}\,$). \end{ass}

We shall use some ideas of \cite{klt2} which, in our setting, provide a quite elementary approach to the problem.
We maintain the notations of section \ref{LamMod}: $\calf/F$ is a $\Z_p^\N$-extension with Galois group $\Gamma$ and associated Iwasawa
algebra $\Lambda$ with augmentation ideal $I$. For any $d\geqslant 0$ let $\calf_d$ be a $\Z_p^d$-extension of $F$ contained in $\calf$,
taken so that $\bigcup \calf_d=\calf$. \bigskip

For any finite extension $L/F$ let $\calm(L)$ be the $p$-adic completion of the group of divisor classes $\Div(L)/P_L$ of $L$, i.e.,
\[ \calm(L):=(L^*\backslash\idl_L/\Pi_v \ol_v^*) \otimes \Z_p \]
where $\idl_L$ is the group of ideles of $L$. As before, when $\call/F$ is an infinite extension,
we put $\calm(\call):=\liminv\calm(K)$ as $K$ runs among finite subextensions of $\call/F$
(the limit being taken with respect to norm maps). For two finite extensions $L\supset L'\supset F$,
the degree maps $\deg_L$ and $\deg_{L'}$ fit into the commutative diagram (with exact rows)
\begin{equation}\label{CommDiagDeg}
\xymatrix{\cala(L)\ar@{^(->}[r] \ar[d]^{N^L_{L'}}& \calm(L) \ar@{->>}[rr]^{\deg_L}\ar[d]^{N^L_{L'}} & & \Z_p \ar[d] \\
\cala(L')\ar@{^(->}[r] & \calm(L') \ar@{->>}[rr]^{\deg_{L'}} & & \Z_p \ ,} \end{equation}
where $N^L_{L'}$ denotes the norm and the vertical map on the right is multiplication by $[\mathbb{F}_L:\mathbb{F}_{L'}]$ (the degree of the
extension between the fields of constants). For an infinite extension $\call/F$ contained in $\calf$, taking projective limits
(and recalling Assumption \ref{assumpt} above), one gets an exact sequence
\begin{equation}\label{degseq}
\xymatrix{ \cala(\call)\ar@{^(->}[r] & \calm(\call) \ar@{->>}[rr]^{\deg_{\call}} & & \Z_p } \ .
\end{equation}

\begin{rem}
If one allows non-geometric extensions, then the $\deg_{\call}$ map above becomes the zero map exactly
when the $\Z_p$-extension $\calf^{arit}\,$ is contained in $\call$.
\end{rem}

It is well known that $\calm({\calf_d})$ is a finitely generated torsion $\L(\calf_d)$-module (see e.g. \cite[Theorem 1]{Gr1}),
so the same holds for $\cala(\calf_d)$
as well. Moreover take any $\Z_p^d$-extension $\calf_d$ of $F$ contained in $\calf$: since our extension $\calf/F$ is totally ramified at
the prime $\pr$, for any $\calf_{d-1}\subset \calf_d$ one has
\begin{equation}\label{isoM} \calm(\calf_d)/I^{\calf_d}_{\calf_{d-1}}\calm(\calf_d)\simeq \calm(\calf_{d-1}) \end{equation}
(see for example \cite[Lemma 13.15]{wash}). As in section \ref{LamMod}, to ease notations we will often erase the $\calf$ from the indices (for example
$I^{\calf_d}_{\calf_{d-1}}$ will be denoted by $I^d_{d-1}\,$), hoping that no confusion will arise. Consider the following diagram
\begin{equation}
\xymatrix{ \cala(\calf_d) \ar@{^(->}[r]\ar[d]^{\gamma-1} & \calm(\calf_d) \ar@{->>}[r]^{\deg}\ar[d]^{\gamma-1} & \Z_p \ar[d]^{\gamma-1} \\
\cala(\calf_d) \ar@{^(->}[r] & \calm(\calf_d) \ar@{->>}[r]^{\deg} & \Z_p }\end{equation}
(where $\overline{\langle\gamma\rangle}=Gal(\calf_d/\calf_{d-1})=:\Gamma^d_{d-1}\,$; note also that the vertical map on the right is 0)
and its snake lemma sequence
\begin{equation} \label{snake0}
\xymatrix{ \cala(\calf_d)^{\Gamma^d_{d-1}} \ar@{^(->}[r] & \calm(\calf_d)^{\Gamma^d_{d-1}} \ar[r]^{\deg} &
\Z_p \ar[d] \\
\Z_p & \calm(\calf_d)/I^d_{d-1}\calm(\calf_d) \ar@{->>}[l]_{\deg\ \quad\ } &  \cala(\calf_d)/I^d_{d-1}\cala(\calf_d) \ar[l]\ .}
\end{equation}
For $d\geqslant 2$ the $\L_d$-module $\Z_p$ is pseudo-null, hence \eqref{degseq} yields $Ch_{\L_d}(\calm(\calf_d))=Ch_{\L_d}(\cala(\calf_d))$,
and, using \eqref{isoM} and \eqref{snake0}, one finds (for $d\geqslant 3$)
\begin{align} \label{CharA} Ch_{\L_{d-1}} (\cala(\calf_d)/I^d_{d-1}\cala(\calf_d))
& = Ch_{\L_{d-1}} (\calm(\calf_d)/I^d_{d-1}\calm(\calf_d)) \nonumber\\
& = Ch_{\L_{d-1}} (\calm(\calf_{d-1})) = Ch_{\L_{d-1}} (\cala(\calf_{d-1})) \end{align}
(where all the modules involved are $\L_{d-1}$-torsion modules).

Let
\begin{equation} \label{exseqNR} N(\calf_d) \iri \cala(\calf_d) \stackrel{\iota}{\longrightarrow} E(\calf_d) \sri R(\calf_d) \end{equation}
be the exact sequence coming from the structure theorem for $\Lambda_d$-modules (see section \ref{LamMod}), where
\[ E(\calf_d):=\bigoplus_i \Lambda_d/(f_{i,d}) \]
is an elementary module and $N(\calf_d)$, $R(\calf_d)$ are pseudo-null. Let $Ch_{\Lambda_d}(\cala(\calf_d))$
be the characteristic ideal of $\cala(\calf_d)$: we want to compare $Ch_{\Lambda_{d-1}}(\cala(\calf_{d-1}))$
with $\pi^{\calf_d}_{\calf_{d-1}}(Ch_{\Lambda_d}(\cala(\calf_d)))$ for some $\calf_{d-1}\subset \calf_d$ and show that these
characteristic ideals form an inverse system
(in $\Lambda$). Consider the module $B(\calf_d):=N(\calf_d)\oplus R(\calf_d)$. We need the following hypothesis.

\begin{ass} \label{assumptFd} There is a choice of the pseudo-isomorphism $\iota$ in \eqref{exseqNR} and a splitting
of the projection $\Gamma_d\twoheadrightarrow\Gamma_{d-1}$ so that
\begin{itemize}
\item[{\it i)}] $\Gamma_d= \overline{\langle\gamma_d\rangle} \oplus \Gamma_{d-1}$;
\item[{\it ii)}] $B(\calf_d)$ is a finitely generated torsion $\Z_p[[\Gamma_{d-1}]]$-module.
\end{itemize}
\end{ass}

As explained in \cite{Gr} (see the remarks just before Lemma 3), for any $\calf_d$ and $\iota$ one can find a subfield $\calf_{d-1}$
so that Assumption \ref{assumptFd} holds.

In order to ease notations, we put  $\gamma=\gamma_d$, so that $\Gamma^d_{d-1}=\overline{\langle\gamma\rangle}$.

\begin{lem}\label{FirstCharIdEq}
With the above notations, one has
\[ Ch_{\Lambda_{d-1}}(\cala(\calf_d)/I^d_{d-1}\cala(\calf_d)) = \pi^d_{d-1}\big(Ch_{\Lambda_d}(\cala(\calf_d))\big)\cdot
Ch_{\Lambda_{d-1}}(\cala(\calf_d)^{\Gamma^d_{d-1}})\ . \]
\end{lem}

\begin{proof}
We split the previous sequence in two by
\[ N(\calf_d) \iri \cala(\calf_d) \sri C(\calf_d) \ ,\
C(\calf_d) \iri E(\calf_d) \sri R(\calf_d)  \]
and consider the snake lemma sequences coming from the following diagrams
\begin{equation} \label{diags}
\xymatrix{ N(\calf_d) \ar@{^(->}[r]\ar[d]^{\gamma-1} & \cala(\calf_d) \ar@{->>}[r]\ar[d]^{\gamma-1} & C(\calf_d) \ar[d]^{\gamma-1} &
C(\calf_d) \ar@{^(->}[r]\ar[d]^{\gamma-1} & E(\calf_d) \ar@{->>}[r]\ar[d]^{\gamma-1} & R(\calf_d) \ar[d]^{\gamma-1} \\
N(\calf_d) \ar@{^(->}[r] & \cala(\calf_d) \ar@{->>}[r] & C(\calf_d) &
C(\calf_d) \ar@{^(->}[r] & E(\calf_d) \ar@{->>}[r]& R(\calf_d) \ ,}\end{equation}
i.e.,
\begin{equation} \label{snake1} \xymatrix{ N(\calf_d)^{\Gamma^d_{d-1}} \ar@{^(->}[r] & \cala(\calf_d)^{\Gamma^d_{d-1}} \ar[r] &
C(\calf_d)^{\Gamma^d_{d-1}} \ar[d] \\
C(\calf_d)/I^d_{d-1}C(\calf_d) & \cala(\calf_d)/I^d_{d-1}\cala(\calf_d) \ar@{->>}[l] &  N(\calf_d)/I^d_{d-1}N(\calf_d) \ar[l]}
\end{equation}
and
\begin{equation} \label{snake2}\xymatrix{ C(\calf_d)^{\Gamma^d_{d-1}} \ar@{^(->}[r] & E(\calf_d)^{\Gamma^d_{d-1}} \ar[r] &
R(\calf_d)^{\Gamma^d_{d-1}} \ar[d] \\
R(\calf_d)/I^d_{d-1}R(\calf_d) & E(\calf_d)/I^d_{d-1}E(\calf_d) \ar@{->>}[l] & C(\calf_d)/I^d_{d-1}C(\calf_d) \ar[l]\ .}
\end{equation}

From \eqref{exseqNR} we get an exact sequence
$$\cala(\calf_d)/I^d_{d-1}\cala(\calf_d)\longrightarrow \bigoplus_i \Lambda_d/(\gamma-1, f_{i,d}) \longrightarrow  R(\calf_d)/I^d_{d-1}R(\calf_d)$$
where the last term is a torsion $\Lambda_{d-1}$-module. So is $\cala(\calf_{d-1})$ for $d\geqslant 3$ and, by \eqref{CharA},
$Ch_{\L_{d-1}} (\cala(\calf_{d-1})) = Ch_{\L_{d-1}} \big(\cala(\calf_d)/I^d_{d-1}\cala(\calf_d)\big)\,$.
It follows that none of the $f_{i,d}$'s belong to $I^d_{d-1}\,$.
Therefore:\begin{itemize}
\item[{\bf 1.}] the map $\gamma -1\,\colon\,E(\calf_d)\longrightarrow E(\calf_d)$ has trivial kernel, i.e.,
$E(\calf_d)^{\Gamma^d_{d-1}}=0$ so that $C(\calf_d)^{\Gamma^d_{d-1}}=0$ as well;
\item[{\bf 2.}] the characteristic ideal of the $\Lambda_{d-1}$-module $E(\calf_d)/I^d_{d-1}E(\calf_d)$ is generated by the
product of the $f_{i,d}$'s modulo $I^d_{d-1}\,$, hence it is obviously equal to $\pi^d_{d-1}(Ch_{\Lambda_d}(\cala(\calf_d)))$.
\end{itemize}
Moreover, from the fact that $N(\calf_d)$ and $R(\calf_d)$ are finitely generated torsion $\Lambda_{d-1}$-modules
\footnote{It might be worth to notice that this is the only point where we use the hypothesis that Assumption \ref{assumptFd} holds.}
and the multiplicativity of characteristic ideals, looking at the left (resp. right) vertical sequence of the first
(resp. second) diagram in \eqref{diags}, one finds
\[ Ch_{\Lambda_{d-1}}(N(\calf_d)^{\Gamma^d_{d-1}})=Ch_{\Lambda_{d-1}}(N(\calf_d)/I^d_{d-1}N(\calf_d))\]
and
\[ Ch_{\Lambda_{d-1}}(R(\calf_d)^{\Gamma^d_{d-1}})=Ch_{\Lambda_{d-1}}(R(\calf_d)/I^d_{d-1}R(\calf_d))\ .\]
Hence from \eqref{snake1} one has
\begin{align} Ch_{\Lambda_{d-1}}(\cala(\calf_d)/I^d_{d-1}\cala(\calf_d))& = Ch_{\Lambda_{d-1}}\big(C(\calf_d)/I^d_{d-1}C(\calf_d)\big)\cdot
Ch_{\Lambda_{d-1}}(N(\calf_d)^{\Gamma^d_{d-1}}) \nonumber\\
& = Ch_{\Lambda_{d-1}}\big(C(\calf_d)/I^d_{d-1}C(\calf_d)\big)\cdot Ch_{\Lambda_{d-1}}(\cala(\calf_d)^{\Gamma^d_{d-1}}) \nonumber \end{align}
(where the last line comes from the isomorphism $\cala(\calf_d)^{\Gamma^d_{d-1}}\simeq N(\calf_d)^{\Gamma^d_{d-1}}\,$).
The sequence \eqref{snake2} provides the equality
\begin{align} Ch_{\Lambda_{d-1}}(C(\calf_d)/I^d_{d-1}C(\calf_d)) & = Ch_{\Lambda_{d-1}}(E(\calf_d)/I^d_{d-1}E(\calf_d))\nonumber\\
& = \pi^d_{d-1}(Ch_{\Lambda_d}(\cala(\calf_d)))\ .\nonumber \end{align}
Therefore one concludes that
\begin{equation}\label{CharId1} Ch_{\Lambda_{d-1}}(\cala(\calf_d)/I^d_{d-1}\cala(\calf_d)) = \pi^d_{d-1}\big(Ch_{\Lambda_d}(\cala(\calf_d))\big)\cdot
Ch_{\Lambda_{d-1}}(\cala(\calf_d)^{\Gamma^d_{d-1}})\ .\end{equation} \end{proof}

Our next step is to prove that $\cala(\calf_d)^{\Gamma^d_{d-1}}=0$ (note that it would be enough to prove that it is pseudo-null
as a $\L_{d-1}$-module). For this we need first a few lemmas.

\begin{lem} \label{tensorcohom}
Let $G$ be a finite group and endow $\Z_p$ with the trivial $G$-action. Then for any $G$-module $M$ we have
$$H^i(G,M\otimes\Z_p)= H^i(G,M)\otimes\Z_p$$
for all $i\geqslant 0$.
\end{lem}

This result should be well-known. Since we were not able to find a suitable reference, here is a sketch of the proof.

\begin{proof}
Let $X$ be a $G$-module which has no torsion as an abelian group and put $Y:=X\otimes\Q$. It is not hard to prove that
$Y^G\otimes\Z_p =(Y\otimes\Z_p)^G$ and it follows that the same holds for $X$,
since $X^G\otimes\Z_p$ is a saturated submodule of $X\otimes\Z_p$. Applying this to the standard complex by means of which
the $H^i(G,M)$ are defined, one can prove the equality in the case $M$
has no torsion as an abelian group. The general case follows because  any $G$-module is the quotient of two such modules.
\end{proof}

Up to now we have mainly considered $\calm(\call)$ as an Iwasawa module (for various $\call$), now we focus on its
interpretation as a group of divisor classes. Let $L$ be a finite extensions of $F$ and recall that we defined $\calm(L)=(\Div(L)/P_L)\otimes\Z_p\,$.
From the exact sequence
\[ \mathbb{F}_L^* \iri L^* \sri P_L \]
and the fact that $|\mathbb{F}_L^*|$ is prime with $p$, one finds an isomorphism between $L^*\otimes\Z_p$ and $P_L\otimes\Z_p\,$.
Hence we can (and will) identify the two.

\begin{lem} \label{divsurj}
For any finite Galois extension $L/F$, the map
\[ \Div(L)^{Gal(L/F)}\otimes\Z_p \longrightarrow \calm(L)^{Gal(L/F)} \]
is surjective.
\end{lem}

\begin{proof}
The sequence \begin{equation}\label{exseqdiv}
L^*\otimes\Z_p \hookrightarrow \Div(L)\otimes\Z_p \twoheadrightarrow \calm(L)
\end{equation}
is exact because $\Z_p$ is flat and $|\mathbb{F}_L^*|$ is prime with $p$. The claim follows by taking the
$Gal(L/F)$-cohomology of \eqref{exseqdiv} and applying Lemma \ref{tensorcohom} and Hilbert 90.
\end{proof}

For any finite subextension $L$ of $\calf/F$, let $\pr_L$ be the unique prime lying above $\pr$. In the following lemma, we identify
$\pr_L$ with its image in $\Div(L)\otimes\Z_p$. Moreover for any element $x\in \calm(\calf)$ we let $x_L$ denote its
image in $\calm(L)$ via the canonical norm map.

\begin{lem} \label{supppL}
Let $x\in \calm(\calf)^{\Gamma}$: then, for any $L$ as above, $x_L$ is represented by a $\Gamma$-invariant divisor supported in $\pr_L\,$.
\end{lem}

\begin{proof}
For any $L$, let $y_L$ be the image of $\pr_L$ in $\calm(L)$. Since $\Z_py_L$ is a closed subset of $\calm(L)$, to prove the lemma
it is enough to show that $\big(x_L+p^n\calm(L)\big)\cap\Z_py_L\neq\emptyset$ for any $n$.

For any finite Galois extension $K/L$ we have the maps
$$\iota_L^K\colon\Div(L)\otimes\Z_p\longrightarrow\Div(K)\otimes\Z_p$$
and
$$N_L^K\colon\Div(K)\otimes\Z_p\longrightarrow\Div(L)\otimes\Z_p$$
respectively induced by the inclusion and the norm. For any divisor whose support is unramifed in $K/L$ we have
$$N_L^K(\iota_L^K(D))=[K:L]D\,.$$
Also, Lemma \ref{tensorcohom} yields
$$(\Div(K)\otimes\Z_p)^{Gal(K/L)}=\Div(K)^{Gal(K/L)}\otimes\Z_p=\iota_L^K(\Div(L)\otimes\Z_p)$$
(since in a $Gal(K/L)$-invariant divisor all places of $K$ above a same place of $L$ occur with the same multiplicity).

Choose $n$ and let $K\subset\calf$ be such that $[K:L]\geqslant p^n$. By Lemma \ref{divsurj}, there exists a $Gal(K/L)$-invariant
$E_K\in\Div(K)\otimes\Z_p$ having image $x_K$. Write $E_K=D_K+a_K\pr_K$, where $a_K\in\Z_p$ and $D_K$ has support disjoint from
$\pr_K$. Then $D_K$ is Galois invariant, so $D_K=\iota_L^K(D_L)$ and (using Assumption \ref{assumpt})
$$N^K_L(E_K)=[K:L]D_L+a_K\pr_L\,.$$
Projecting into $\calm(L)$ we get $x_L\in a_Ky_L+p^n\calm(L)\,.$
\end{proof}

\begin{cor} $\cala(\calf_d)^{\Gamma^d_{d-1}}=0$.
\end{cor}

\begin{proof}
Taking $\Gamma^d_{d-1}$-invariants in \eqref{degseq} (with $\call=\calf_d$), one finds a similar sequence
\begin{equation}\label{ExSeqAGamma} \xymatrix{\cala(\calf_d)^{\Gamma^d_{d-1}}\,\ar@{^(->}[r] &
\calm(\calf_d)^{\Gamma^d_{d-1}} \ar[rr]^{\deg_{\calf_d}} & & \Z_p \ .}\end{equation}
Lemma \ref{supppL} holds, with exactly the same proof, also replacing $\calf$ and $\Gamma$ with $\calf_d$ and $\Gamma_d$.
Therefore any $x=(x_L)_L\in \calm(\calf_d)^{\Gamma^d_{d-1}}$ can be represented by a sequence $(a_L\pr_L)_L$.
Furthermore $N^K_L(a_K\pr_K)=a_L\pr_L$ implies that the value $a_L$ is independent of $L$: call it $a$. Then
$$\deg_{\calf_d}(x)=\lim \big(a_L\deg_L(\pr_L)\big)=a\deg_F(\pr)\,.$$
Hence $x\in Ker(\deg_{\calf_d})=\cala(\calf_d)^{\Gamma^d_{d-1}}$ only if $a=0$.
\end{proof}

\begin{rem}
The image of the degree map appearing in \eqref{ExSeqAGamma} is $(\deg \pr)\Z_p\,$, so $\deg_{\calf_d}$ always provides an isomorphism
between $\calm(\calf_d)^{\Gamma^d_{d-1}}$ and $\Z_p\,$. Moreover, if $p$ does not divide $\deg\pr$, one has
surjectivity as well. In this case, looking back at the sequence \eqref{snake0}, one finds a short exact sequence
\[ \xymatrix{ \cala(\calf_d)/I^d_{d-1} \cala(\calf_d) \ar@{^(->}[r] &
\calm(\calf_d)/I^d_{d-1} \calm(\calf_d) \ar@{->>}[r]^{\qquad\ \deg} & \Z_p }\ .\]
From \eqref{CommDiagDeg}, by taking the limit with $L$ and $L'$ varying respectively among subextensions of $\calf_d$ and $\calf_{d-1}$,
one obtains a commutative diagram
\[ \xymatrix{ \calm(\calf_d)/I^d_{d-1} \calm(\calf_d) \ar@{->>}[rr]^{\qquad\ \deg} \ar[d]^{N} & & \Z_p \ar@{=}[d] \\
\calm(\calf_{d-1})\ar@{->>}[rr]^{\qquad\ \deg_{\calf_{d-1}}} & & \Z_p } \]
where the map $N$ is the isomorphism induced by the norm, i.e., the one appearing in \eqref{isoM}.
This and the exact sequence \eqref{degseq} for $\call=\calf_{d-1}\,$, show that
$\cala(\calf_d)/I^d_{d-1} \cala(\calf_d)\simeq \cala(\calf_{d-1})$ (for any $d\geqslant 1$).
\end{rem}

From \eqref{CharId1} one finally obtains
\begin{equation}\label{CharId2}
Ch_{\Lambda_{d-1}}(\cala(\calf_{d-1}))=Ch_{\Lambda_{d-1}}(\cala(\calf_d)/I^d_{d-1}\cala(\calf_d)) = \pi^d_{d-1}(Ch_{\Lambda_d}(\cala(\calf_d)))\ .
\end{equation}

We remark that this equation holds for any $\Z_p$-extension $\calf_d/\calf_{d-1}$ satisfying Assumption \ref{assumptFd}.
If the filtration $\{\calf_d\,:\,d\in \N\}$ verifies that Assumption at any level $d$, then the inverse images of the
$Ch_{\Lambda(\calf_d)}(\cala(\calf_d))$ in $\Lambda$ (with respect to the canonical projections
$\pi_{\calf_d}:\Lambda\rightarrow \Lambda(\calf_d)\,$) form an inverse system and we can define

\begin{df}\label{CharIdClGr}
The \emph{pro-characteristic ideal} of $\cala(\calf)$ is
\[ \widetilde{Ch}_{\Lambda}(\cala(\calf)):=\il{\calf_d} (\pi_{\calf_d})^{-1}(Ch_{\Lambda(\calf_d)}(\cala(\calf_d)))\ .\]
\end{df}

\begin{rem}
Two questions naturally arise from the above definition:\begin{itemize}
\item[{\it a.}] {\it is there a filtration verifying Assumption \ref{assumptFd} at any level $d$?}
\item[{\it b.}] (assuming {\it a} has a positive answer) {\it is the limit independent from the chosen filtration?}
\end{itemize}
In the next section we are going to show (in particular in \eqref{CharIdStick} and Corollary \ref{corIMC}) that there
is an element $\theta\in\Lambda$, independent of the filtration and such that, for all $\calf_d$, its image in
$\Lambda(\calf_d)$ generates $Ch_{\Lambda(\calf_d)}(\cala(\calf_d))$. Hence Question {\it b} has a positive answer
(and presumably so does Question {\it a})
and we only needed \eqref{CharId2} as a first step and a natural analogue of \eqref{Stick1}.\\
Nevertheless we believe that these questions have some interest on their own and it would be nice to have a direct construction
of a ``good'' filtration $\{\calf_d\,:\,d\in \N\}$ based on a generalization of \cite[Lemma 2]{Gr}. Since our goal here is the
Main Conjecture we do not pursue this subject further, but we
hope to get back to it in a future paper.\\
We also observe that Assumption \ref{assumptFd} was used only in one passage in the proof of Lemma \ref{FirstCharIdEq},
as we evidentiated in a footnote. It might be easier to show that in that passage one does not need the finitely generated
hypothesis: if so, Definition \ref{CharIdClGr} would makes sense for all filtrations $\{\calf_d\}_d$\ \footnote{This is
exactly the approach taken in \cite{BBL1}, providing a positive answers to question {\it b}.}.
\end{rem}

\begin{rem}
We could have used Fitting ideals, just as we did in section \ref{LamMod}, to provide a more straightforward construction (there would have
been no need for preparatory lemmas). But, since the goal is a Main Conjecture, the characteristic ideals, being principal, provide a better
formulation. We indeed expect equality between Fitting and characteristic ideals in all the cases studied in this paper but, at present,
are forced to distinguish between them (but see Remark \ref{RemBLT}).
\end{rem}

\subsection{Stickelberger elements} \label{s:stickelberger}
We shall briefly describe a relation between the characteristic ideal of the previous section and Stickelberger elements. The main results
on those elements are due to A. Weil, P. Deligne and J. Tate and for all the details the reader can consult \cite[Ch. V]{tate}. Let $S$ be
a finite set of places of $F$ containing all places where the extension $\calf/F$ ramifies; since we are interested in the case where $\calf$
is substantially bigger than the arithmetic extension, we assume $S\neq\emptyset$. We consider also another non-empty finite set $T$ of places
of $F$ such that $S\cap T=\emptyset$. For any place outside $S$ let $Fr_v$ be the Frobenius of $v$ in $\Gm=Gal(\calf/F)$.

Let
\begin{equation}\label{e:Thetaprod}
\Theta_{\calf/F,S,T}(u):=\prod_{v\in T} (1-Fr_vq^{\deg(v)}u^{\deg(v)})\prod_{v\not\in S} (1-Fr_vu^{\deg(v)})^{-1}.
\end{equation}

For any $n\in\N$ there are only finitely many places of $F$ with degree $n$: hence we can expand \eqref{e:Thetaprod} and consider
$\Theta_{\calf/F,S,T}(u)$ as a power series $\sum c_nu^n\in\Z[\Gm][[u]]$. Moreover, it is clear that for any continuous character
$\psi\colon\Gamma\longrightarrow\C^*$ the image $\psi(\Theta_{\calf/F,S,T}(q^{-s}))$ is the $L$-function of $\psi$, relative to $S$
and modified at $T$. For any subextension $F\subset L\subset\calf$, let $\pi^\calf_L\colon\Z[\Gm][[u]]\rightarrow \Z[Gal(L/F)][[u]]$
be the natural projection and define
$$\Theta_{L/F,S,T}(u):=\pi^\calf_L(\Theta_{\calf/F,S,T}(u)).$$
For $L/F$ finite it is known (essentially by Weil's work) that $\Theta_{L/F,S,T}(q^{-s})$ is an element in the polynomial ring $\C[q^{-s}]$
(see \cite[Ch.\,V,  Proposition 2.15]{tate} for a proof): hence $\Theta_{L/F,S,T}(u)\in\Z[Gal(L/F)][u].$ It follows that the coefficients
$c_n$ of $\Theta_{\calf/F,S,T}(u)$ tend to zero in
$$\liminv\Z[Gal(L/F)]=:\Z[[\Gm]]\subset\Lambda\,.$$
Therefore we can define
$$\theta_{\calf/F,S,T}:=\Theta_{\calf/F,S,T}(1)\in\Lambda\,.$$
We also observe that the factors $(1-Fr_vq^{\deg(v)}u^{\deg(v)})$ in \eqref{e:Thetaprod} are units in the ring $\L[[u]]$. Hence the ideal
generated by $\theta_{\calf/F,S,T}$ is independent of the auxiliary set $T$ and we can define the Stickelberger element
$$\theta_{\calf/F,S}:=\theta_{\calf/F,S,T}\prod_{v\in T}(1-Fr_vq^{\deg(v)})^{-1}.$$ \index{Stickelberger element}
We also define, for $F\subset L\subset\calf$,
$$\theta_{L/F,S,T}:=\pi^\calf_L(\theta_{\calf/F,S,T})=\Theta_{L/F,S,T}(1).$$
It is clear that these form a projective system: in particular, for any $\Z_p$-extension $\calf_d/\calf_{d-1}$ the relation
\begin{equation}\label{Stick1} \pi^d_{d-1}(\theta_{\calf_d/F,S,T})=\theta_{\calf_{d-1}/F,S,T} \end{equation}
clearly recalls the one satisfied by characteristic ideals (equation \eqref{CharId2}). Also, to define $\theta_{L/F,S}$ there is no need of
$\calf$: one can take  for a finite extension $L/F$ the analogue of product \eqref{e:Thetaprod} and reason as above.

\begin{thm}[Tate, Deligne]
For any finite extension $L/F$, $|\F_L^*|\theta_{L/F,S}$ is in the annihilator ideal of the class group of $L$ (considered
as a $\Z[Gal(L/F)]$-module).
\end{thm}

\begin{proof}
This is \cite[Ch.\,V, Th\'eor\`eme 1.2]{tate}.
\end{proof}

\begin{rem} Another proof of this result was given by Hayes \cite{hayes}, by means of Drinfeld modules. \end{rem}

\begin{cor} Let $\calf_d/F$ be a $\Z_p^d$-extension as before and $S=\{\pr\}$, the unique (totally) ramified prime in $\calf/F$: then
\begin{itemize}
\item[{\bf 1.}] $\theta_{\calf_d/F,S}\cala(\calf_d)=0$;
\item[{\bf 2.}] if $\theta_{\calf_d/F,S}$ is irreducible in $\Lambda(\calf_d)$, then $Ch_{\Lambda(\calf_d)}(\cala(\calf_d))=
(\theta_{\calf_d/F,S})^m$ for some $m\geqslant 1$;
\item[{\bf 3.}] if $\theta_{\calf_d/F,S}$ is irreducible in $\Lambda(\calf_d)$ for all $\calf_d$'s,
then $\wt{Ch}_\Lambda(\cala(\calf))=(\theta_{\calf/F,S})^m$ for some $m\geqslant 1$.
\end{itemize}
\end{cor}

\begin{proof}
For {\bf 1} one just notes that $|\F_L^*|$ is prime with $p$. Part {\bf 2} follows from the structure theorem for torsion
$\Lambda(\calf_d)$-modules. Part {\bf 3} follows from {\bf 2} by taking limits (as in Definition \ref{CharIdClGr}) and noting that the $m$
is constant through the $\calf_d$'s because of equations \eqref{CharId2} and \eqref{Stick1}.
\end{proof}

The exponent in {\bf 2} and {\bf 3} of the corollary above is actually $m=1$.
A proof of this fact is based on the following technical result of \cite{klt} (generalized in \cite[Theorem A.1]{blt09A}). Once $\calf_d$
is fixed it is always possible to find a $\Z_p^c$-extension of $F$ containing $\calf_d$, call it $\call_d$, such that:\begin{itemize}
\item[{\bf a.}] the extension $\call_d/F$ is ramified at all primes of a finite set $\widetilde{S}$ containing $S$ (moreover
$\widetilde{S}$ can be chosen arbitrarily large);
\item[{\bf b.}] the Stickelberger element $\theta_{\call_d/F,\widetilde{S}}$ is irreducible in the Iwasawa algebra $\Lambda(\call_d)$;
\item[{\bf c.}] there is a $\Z_p$-extension $\call'$ of $F$ contained in $\call_d$ which is ramified at all primes of
$\widetilde{S}$ and such that the Stickelberger element $\theta_{\call'/F,\widetilde{S}}$ is monomial, i.e., congruent to $u(\sigma-1)^r$
modulo $(\sigma -1)^{r+1}$ (where $\sigma$ is a topological generator of $Gal(\call'/F)$ and $u\in \Z_p^*\,$).
\end{itemize}

With condition {\bf b} and an iteration of equation \eqref{CharId2} one proves that
\[ Ch_{\Lambda(\calf_d)}(\cala(\calf_d))= (\theta_{\calf_d/F,S})^m\ {\rm for\ some\ } m\geqslant 1 \ .\]
The monomiality condition {\bf c} (using $\call'$ as a first layer in a tower of $\Z_p$-extensions) leads to $m=1$
(see \cite[section 4]{klt2} or \cite[section A.1]{blt09A} which uses the possibility of varying the set $\widetilde{S}$, provided by
{\bf a}, more directly). We remark that the proof only uses the irreducibility of $\theta_{\call_d/F,\widetilde{S}}$, i.e.,
\begin{equation}\label{CharIdStick} Ch_{\Lambda(\calf_d)}(\cala(\calf_d))= (\theta_{\calf_d/F,S}) \end{equation}
holds in general for any $\calf_d\,$.

\begin{cor}[Iwasawa Main Conjecture] \label{corIMC}  In the previous setting we have
\[ \widetilde{Ch}_{\L}(\cala(\calf)) = (\theta_{\calf/F,\pr})  \ .\]
\end{cor}

\begin{proof} From the main result of \cite{klt2}, one has that
\[ Ch_{\L(\calf_d)}(\cala(\calf_d)) = (\theta_{\calf_d/F,\pr}) \]
and we take the limit in both sides.
\end{proof}

\begin{rem}\label{RemBLT}
The equality between characteristic ideals and ideals generated by Stickelberger elements has been proved by K.-L. Kueh, K. F. Lai and
K.-S. Tan (\cite{klt2}) and by D. Burns (\cite{blt09} and the Appendix coauthored with K.F. Lai and K-S. Tan \cite{blt09A}) in a more
general situation. The $\Z_p^d$-extension they consider has to be unramified outside a finite set $S$ of primes of $F$ (but there is no need
for the primes to be totally ramified). Moreover they require that none of the primes in $S$ is totally split (otherwise
$\theta_{\calf_d/F,S}=0$). The strategy of the proof is basically the same but, of course, many technical details are simplified by our choice
of having just one (totally) ramified prime (just compare, for example, Lemma \ref{supppL} with \cite[Lemma 3.3 and 3.4]{klt2}). Moreover,
going back to the Fitting vs. characteristic ideal situation, it is worth noticing that Burns proves that the first cohomology group of
certain complexes (strictly related to class groups, see \cite[Proposition 4.4]{blt09} and \cite[section A.1]{blt09A}) are of projective
dimension 1 (\cite[Proposition 4.1]{blt09}). In this case the Fitting and characteristic ideals are known to be equal to the inverse of
the Knudsen-Mumford determinant (the ideal by which all the results of \cite{blt09} are formulated).
\end{rem}

\subsection{Characteristic $p$ $L$-functions} \label{s:gosszeta} One of the most fascinating aspects of function field arithmetic is the
existence, next to complex and $p$-adic $L$-functions, of their characteristic $p$ avatars. For a thorough introduction the reader is referred
to \cite[Chapter 8]{goss}: here we just provide a minimal background.

Recall our fixed place $\infty$ and let $\dc$ denote the completion of an algebraic closure of $F_\infty$. Already Carlitz had studied
a characteristic $p$ version of the Riemann zeta function, defined on $\N$ and taking values in $\dc$ (we will say more about it in section
\ref{s:berncarl}). More recently Goss had the intuition that, like complex and $p$-adic $L$-functions have as their natural domains
respectively the complex and the $p$-adic (quasi-)characters of the Weil group, so one could consider $\dc$-valued characters. In particular,
the analogue of the complex plane as domain for the characteristic $p$ $L$-functions is $S_\infty:=\dc^*\times\Z_p$, that can be seen as a group
of $\dc$-valued homomorphisms on $F_\infty^*$, just as for $s\in\C$ one defines $x\mapsto x^s$ on $\R^+$. The additive group $\Z$ embeds
discretely in $S_\infty$. Similarly to the classical case, one can define $L(\rho,s)$ for $\rho$ a compatible system of $v$-adic representation
of $G_F$ ($v$ varying among places different from $\infty$) by Euler products converging on some ``half-plane'' of $S_\infty$.

The theory of zeta values in characteristic $p$ is still quite mysterious and at the moment we can at best speculate that there are links with
the Iwasawa theoretical questions considered in this paper\,\footnote{This field is in rapid evolution.
After this paper was written, L. Taelman introduced some important new ideas: see  \cite{Tael1} and \cite{Tael2}.
Further recent developments can be found in \cite{ABBL}, \cite{AP}, \cite{AP2}, \cite{AT}, \cite{AT2} and \cite{Pel}.}.
To the best of our knowledge, the main results available in this direction are
the following. Let $F(\pr)/F$ be the extension obtained from the $\pr$-torsion of a Drinfeld-Hayes module (in the simplest case, $F(\pr)$ is
the $F_1$ we are going to introduce in section \ref{CycNot}). Goss and Sinnott have studied the isotypic components of $\cala(F(\pr))$ and shown
that they are non-zero if and only if $\pr$ divides certain characteristic $p$ zeta values: see \cite[Theorem 8.14.4]{goss} for a precise
statement. Note that the proof given in \cite{goss}, based on a comparison between the reductions of a $p$-adic and a characteristic $p$
$L$-function respectively mod $p$ and mod $\pr$ (\cite[Theorem 8.13.3]{goss}), makes use of Crew's result.
Okada \cite{Okada} obtained a result of similar
flavor for the class group of the ring of ``integers'' of $F(\pr)$ when $F$ is the rational function field, and Shu \cite{shu}
extended it to any $F\,$; since Okada's result is strictly related with the subject of section \ref{s:berncarl} below,
we will say more about it there.

\section{Cyclotomy by the Carlitz module} \label{s:carlitz}

\subsection{Setting} \label{CycNot} From now on we take $F=\F_q(T)$ and let $\infty$ be the usual place at infinity, so that the ring of
elements regular outside $\infty$ is $A:=\F_q[T]$; this allows a number of simplifications, leaving intact the main aspects of the theory.
The ``cyclotomic'' theory of function fields is obtained via Drinfeld-Hayes modules: in the setting of the rational function field the only
one is the Carlitz module $\Phi\colon A\rightarrow A\{\tau\}$, $T\mapsto\Phi_T:=T+\tau$ (here $\tau$ denotes the operator $x\mapsto x^q$ and,
if $R$ is an $\F_p$-algebra, $R\{\tau\}$ is the ring of skew polynomials with coefficients in $R$: multiplication in $R\{\tau\}$ is given
by composition).

We also fix a prime $\pr\subset A$ and let $\pi\in A$ be its monic generator. In order to underline the fact that $A$ and its completion at
$\pr$ play the role of $\Z$ and $\Z_p$ in the Drinfeld-Hayes cyclotomic theory, we will often use the alternative notation $A_\pr$ for the
ring of local integers $\ol_\pr\subset F_\pr$. Let $\cp$ be the completion of an algebraic closure of $F_\pr$.

As usual, if $I$ is an ideal of $A$, $\Phi[I]$ will denote the $I$-torsion of $\Phi$ (i.e., the common zeroes of all $\Phi_a$, $a\in I$).
One checks immediately that if $\iota$ is the unique monic generator of $I$, then
$$\Phi_\iota(x)=\prod_{u\in\Phi[I]}(x-u)\ .$$
We put
$$F_n:=F(\Phi[\pr^n])$$ and
$$K_n:=F_\pr(\Phi[\pr^n])\,.$$
As stated in section \ref{notation1}, we think of the $F_n$'s as subfields of $\cp$, so that the $K_n$'s are their topological closures. We shall denote
the ring of $A$-integers in $F_n$ by $B_n$ and its closure in $K_n$ by $\ol_n$, and write $\calu_n$ for the 1-units in $\ol_n$. Let
$\calf:=\cup F_n$ and $\tilde\Gamma:=Gal(\calf/F)$.

Consider the ring of formal skew power series $A_\pr\{\{\tau\}\}$: it is a complete local ring, with maximal ideal
$\pi A_\pr+A_\pr\{\{\tau\}\}\tau$. It is easy to see that $\Phi$ extends to a continuous homomorphism
$\Phi\colon A_\pr\rightarrow A_\pr\{\{\tau\}\}$ (i.e., a formal Drinfeld module) and this allows to define a ``cyclotomic"
character $\chi\colon\tilde\Gamma\stackrel{\sim}{\longrightarrow}A_\pr^*$. More precisely, let $T_\pr\Phi:=\liminv \Phi[\pr^n]$
(the limit is taken with respect to $x\mapsto\Phi_{\pi}(x)$) be the Tate module of $\Phi$. The ring $A_\pr$ acts on $T_\pr\Phi$ via
$\Phi$, i.e., $a\cdot(u)_n:=(\Phi_a(u_n))_n$, and the character $\chi$ is defined by $\sigma u=:\chi(\sigma)\cdot u$, i.e., $\chi(\sigma)$
is the unique element in $A_\pr^*$ such that $\Phi_{\chi(\sigma)}(u_n)=\sigma u_n$ for all $n$. From this it follows immediately
that $\tilde\Gamma=\Delta\times\Gamma$, where $\Delta\simeq\F_\pr^*$ is a finite group of order prime to $p$ and $\Gamma$ is the inverse
image of the 1-units.

Since $\Phi$ has rank 1, $T_\pr\Phi$ is a free $A_\pr$-module of rank 1. As in \cite{fbl}, we fix a generator $\omega=(\omega_n)_{n\geqslant 1}$:
this means that the sequence $\{\omega_n\}$ satisfies
$$\Phi_{\pi^n}(\omega_n)=0\neq\Phi_{\pi^{n-1}}(\omega_n)\quad{\rm and}\quad\Phi_\pi(\omega_{n+1})=\omega_n\ .$$

By definition $K_n=F_\pr(\omega_n)$. By Hayes's theory, the minimal polynomial of $\omega_n$ over $F$ is Eisenstein: it follows that the
extensions $F_n/F$ and $K_n/F_\pr$ are totally ramified, $\omega_n$ is a uniformizer for the field $K_n$,
$\ol_n=A_\pr[[\omega_n]]=A_\pr[\omega_n]$. The extension $F_n/F$ is unramified at all other finite places: this can be seen directly by
observing that $\Phi_{\pi^n}$ has constant coefficient $\pi^n$. Furthermore $F_n/F$ is tamely ramified at $\infty$ with inertia
group $I_\infty(F_n/F)\simeq\F_q^*\,$.

The similarity with the classical properties of $\Q(\zeta_{p^n})/\Q$ is striking.\bigskip

The formula $N_{F_{n+1}/F_n}(\omega_{n+1})=\omega_n$ shows that the $\omega_n$'s form a compatible system under the norm maps (the proof
is extremely easy; it can be found in \cite[Lemma 2]{fbl}). This and the observation that $[F_{n+1}:F_n]=q^{\deg(\pr)}$ for $n\geqslant 1$
imply
\begin{equation} \label{limK} \liminv K_n^*=\omega^{\Z}\times\F_\pr^*\times\liminv\calu_n\ . \end{equation}
Note that $\liminv\calu_n$ is a $\tilde\L$-module.

\subsection{Coleman's theory} A more complete discussion and proofs of results in this section can be found in \cite[\S3]{fbl}. Let $R$ be a
subring of $\cp$: then, as usual, $R((x)):=R[[x]](x^{-1})$ is the ring of formal Laurent series with coefficients in $R$. Moreover, following
\cite{col} we define $R[[x]]_1$ and $R((x))_1$ as the subrings consisting of those (Laurent) power series which converge on the punctured open
ball
$$B':=B(0,1)-\{0\}\subset\cp\ .$$
The rings $R[[x]]_1$ and $R((x))_1$ are endowed with a structure of topological $R$-algebras, induced by the family of seminorms
$\{\|\cdot\|_r\}$, where $r$ varies in $|\cp|\cap(0,1)$ and $\|f\|_r:=\sup\{|f(z)|:|z|=r\}$.\bigskip

All essential ideas for the following two theorems are due to Coleman \cite{col}.

\begin{thm} \label{teo4col} There exists a unique continuous homomorphism
$$\caln:F_\pr((x))_1^*\rightarrow F_\pr((x))_1^*$$
such that
$$\prod_{u\in\Phi[\pr]}f(x+u)=(\caln f)\circ\Phi_\pi \ .$$
\end{thm}

\begin{thm} \label{teo8} The evaluation map $ev:f\mapsto\{f(\omega_n)\}$ gives an isomorphism
$$(A_\pr((x))^{\ast})^{\caln=id}\simeq\liminv K_n^{\ast}$$
where the inverse limit is taken with respect to the norm maps.
\end{thm} \index{Coleman power series}

We shall write $Col_u$ for the power series in $A_\pr((x))^{\ast}$ associated to $u\in\liminv K_n^{\ast}$ by Coleman's isomorphism of
Theorem \ref{teo8}.

\begin{rem} \label{Phiacol} An easily obtained family of $\caln$-invariant power series is the following. Let $a\in A_\pr^*$: then
$$\prod_{u\in\Phi[\pr]}\Phi_a(x+u)=\prod_{u\in\Phi[\pr]}(\Phi_a(x)+\Phi_a(u))=\Phi_\pi(\Phi_a(x))$$
(since $\Phi_a$ permutes elements in $\Phi[\pr]$) and, from $\Phi_\pi\Phi_a=\Phi_a\Phi_\pi$ in $A_\pr\{\{\tau\}\}$, it follows that
$\Phi_a(x)$ is invariant under the Coleman norm operator $\caln$ (as observed in \cite[page 797]{fbl}, this just amounts to replacing
$\omega$ with $a\cdot\omega$ as generator of the Tate module). \end{rem}

Following \cite{col}, we define an action of $\Gamma$ on $F_\pr[[x]]_1$ by $(\sigma*f)(x):=f(\Phi_{\chi(\sigma)}(x))$. Then
$Col_{\sigma u}=(\sigma*Col_u)$, as one sees from
\begin{equation} \label{sigmacol} (\sigma*Col_u)(\omega_n)=Col_u(\Phi_{\chi(\sigma)}
(\omega_n))=Col_u(\sigma\omega_n)=\sigma(Col_u(\omega_n))=\sigma(u_n)\ . \end{equation}

\subsection{The Coates-Wiles homomorphisms} We introduce some operators on power series. Let
$\dlog\colon F_\pr((x))_1^{\ast}\rightarrow F_\pr((x))_1$ be the logarithmic derivative, i.e., $\dlog(g):=\frac{g'}{g}$.
Also, for any $j\in\N$ let $\Delta_j\colon F_\pr((x))\rightarrow F_\pr((x))$ be the $j$th Hasse-Teichm\"uller derivative,
defined by the formula
$$\Delta_j\left(\sum_{n=0}^{\infty}c_nx^n\right):=\sum_{n=0}^{\infty}\begin{pmatrix} n+j\\
j\end{pmatrix}c_{n+j}x^n$$
(i.e., $\Delta_j$ ``is'' the differential operator $\frac{1}{j!}\frac{{\rm d}^j}{{\rm d}x^j}$). A number of properties of the
Hasse-Teichm\"uller derivatives can be found in \cite{jks}; here we just recall that the operators $\Delta_j$ are $F_\pr$-linear and that
\begin{equation} \label{eq:deltaserie} f(x)=\sum_{j=0}^\infty \Delta_j(f)_{|x=0}\,x^j \ . \end{equation}
 The last operator we need to introduce is composition with the Carlitz exponential $e_C(x)=x+\dots$, i.e., $f\mapsto f(e_C(x))$.

\begin{df} \index{Coates-Wiles homomorphism} For any integer $k\geqslant 1$, define the $k$th Coates-Wiles
homomorphism $\delta_k\colon\liminv\ol_n^*\rightarrow F_\pr$ by
$$\delta_k(u):=\Delta_{k-1}\big((\dlog Col_u)(e_C(x))\big)_{|x=0}=\big(\Delta_{k-1}((\dlog Col_u)\circ e_C)\big)(0)\ .$$
\end{df}

Notice that by \eqref{eq:deltaserie} this is equivalent to putting
\begin{equation} \label{eq:deltaexp} (\dlog Col_u)(e_C(x))=\sum_{k=1}^\infty \delta_k(u)x^{k-1}\ . \end{equation}

\begin{lem} The Coates-Wiles homomorphisms satisfy
$$\delta_k(\sigma u)=\chi(\sigma)^k\delta_k(u)\ .$$
\end{lem}

\begin{proof} Recall that $\frac{{\rm d}}{{\rm d}x}\Phi_a(x)=a$ for any $a\in A_\pr$. Then from \eqref{sigmacol} it follows
$$\dlog Col_{\sigma u}=\dlog(Col_u\circ\Phi_{\chi(\sigma)})=\chi(\sigma)(\dlog Col_u)\circ\Phi_{\chi(\sigma)}\ ,$$
since $\dlog(f\circ g)=g'(\frac{f'}{f}\circ g)$.
Composing with $e_C$ and using $\Phi_a(e_C(x))=e_C(ax)$, one gets, by \eqref{eq:deltaexp},
$$(\dlog Col_{\sigma u})(e_C(x))=\chi(\sigma)(\dlog Col_u)(e_C(\chi(\sigma)x))=
\chi(\sigma)\sum_{k=1}^\infty\delta_k(u)\chi(\sigma)^{k-1}x^{k-1}\ .$$
The result follows.
\end{proof}

\subsection{Cyclotomic units} \label{s:cycunit}

\begin{df} \index{Cyclotomic units} The group $C_n$ of cyclotomic units in $F_n$ is the intersection of $B_n^*$ with the subgroup of
$F_n^*$ generated by $\sigma(\omega_n)$, $\sigma\in Gal(F_n/F)$.
\end{df}

By the explicit description of the Galois action via $\Phi$, one sees immediately that this is the same
as $B_n^*\cap\langle\Phi_a(\omega_n)\rangle_{a\in A-\pr}\,$.

\begin{lem} \label{l:augm} Let $\sum c_\sigma\sigma$ be an element in $\Z[Gal(F_n/F)]$: then
$$\prod_{\sigma\in Gal(F_n/F)}\sigma(\omega_n)^{c_\sigma}\in C_n\Longleftrightarrow\sum c_\sigma=0\ .$$
\end{lem}

\begin{proof} Obvious from the observation that $\omega_n$ is a uniformizer for the place above $\pr$ and a unit at every other finite
place of $F_n$. \end{proof}

Let $\calc_n$ and $\calc_n^1$ denote the closure respectively of $C_n\cap\ol_n^*$ and of $C^1_n:=C_n\cap\calu_n\,$.\bigskip

Let $a\in A_\pr^*$. By Remark \ref{Phiacol} $(\Phi_a(\omega_n))_n$ is a norm compatible system: hence one can define a homomorphism
$$\Upsilon\colon\Z[\tilde\Gamma]\longrightarrow\liminv K_n^*$$
$$\sum c_\sigma\sigma\mapsto\prod\big(\sigma(\omega_n)^{c_\sigma}\big)_n=\prod\big(\Phi_{\chi(\sigma)}(\omega_n)^{c_\sigma}\big)_n\ .$$
Let $\widehat{\liminv K_n^*}$ be the $p$-adic completion of $\liminv K_n^*$. By \eqref{limK} one gets the isomorphism
$\widehat{\liminv K_n^*}\simeq\omega^{\Z_p}\times\liminv\calu_n\,$.

\begin{lem}
The restriction of $\Upsilon$ to $\Z[\Gamma]$ can be extended to $\Upsilon\colon\Lambda\rightarrow\widehat{\liminv K_n^*}\,$.
\end{lem}

\begin{proof} If $a\in A_\pr$ is a 1-unit, then
\begin{equation} \label{e:phia} \Phi_a(\omega_n)=\omega_n u_n \end{equation}
with $u_n\in\calu_n$. Since by definition $\Gamma=\chi^{-1}(1+\pi A_\pr)$, it follows that $\Upsilon$ sends $\Z[\Gamma]$
into $\omega^\Z\times\liminv\calu_n$. To complete the proof it suffices to check that $\Upsilon$ is continuous with respect to the
natural topologies on $\L=\liminv(\Z/p^n\Z)[Gal(F_n/F)]$ and $\widehat{\liminv K_n^*}$. But $a\equiv a'\pmod{\pi^n}$ in $A_\pr$
implies $\Phi_a(\omega_j)=\Phi_{a'}(\omega_j)$ for any $j\leqslant n$ and the result follows from the continuity of $\chi$.
\end{proof}

\begin{prop} \label{p:upsilon} Let $I\subset\Lambda$ denote the augmentation ideal; then $\Upsilon$ induces a surjective homomorphism
of $\L$-modules $I\longrightarrow\liminv\calc_n^1$. The kernel has empty interior.   \end{prop}

\begin{proof} From Lemma \ref{l:augm} and \eqref{e:phia} it is clear that $\Upsilon(\alpha)\in\liminv\calc_n^1$ if and only if $\alpha\in I$.
This map is surjective because $I$ is compact and already the restriction to the augmentation ideal of $\Z[\Gamma]$ is onto $C_n^1$ for all $n$.
A straightforward computation shows that it is a homomorphism of $\L$-algebras: for $\gamma\in\Gamma$
\begin{equation} \label{e:lambdahom} \gamma\Upsilon\left(\sum c_\sigma\sigma\right)=
\left(\gamma\left(\prod\Phi_{\chi(\sigma)}(\omega_n)^{c_\sigma}\right)_n\right)=
\left(\prod\Phi_{\chi(\sigma\gamma)}(\omega_n)^{c_\sigma}\right)_n \end{equation}
because $\gamma(\Phi_a(\omega_n))=\Phi_a(\gamma(\omega_n))=\Phi_a(\Phi_{\chi(\gamma)}(\omega_n))$.

For the statement about the kernel, let $A^+\subset A$ be the subset of monic polynomials and consider any function
$A^+\longrightarrow\Z$, $a\mapsto n_a\,$, such that $n_a=0$ for almost all $a$. We claim that $\prod_{a\in A^+}\Phi_a(x)^{n_a}=1$ only
if $n_a=0$ for all $a$. To see it, let $u_a$ denote a generator of the cyclic $A$-module $\Phi[(a)]$. Then $x-u_a$ divides $\Phi_b(x)$
if and only if $b\in(a)$: hence the multiplicity $m_a$ of $u_a$ as root of $\prod \Phi_a(x)^{n_a}=1$ is exactly $\sum_{b\in(a)\cap A^+}n_b\,$.
For $b\in A$, let $\varepsilon(b)$ denote the number of primes of $A$ dividing $b$ (counted with multiplicities): then a simple combinatorial
argument shows that
$$n_a=\sum_{b\in A^+}(-1)^{\varepsilon(b)}\sum_{c\in A^+}n_{abc}\ .$$
It follows that $m_a=0$ for all $a\in A^+$ if and only if $n_a=0$ for all $a$.

As in section \ref{s:stickelberger}, for $v\neq\pr,\infty$ let $Fr_v\in\tilde\Gamma$ be its Frobenius. By \cite[Proposition 7.5.4]{goss} one finds
that $\chi(Fr_v)$ is the monic generator of the ideal in $A$ corresponding to the place $v$: hence, by Chebotarev density theorem,
$\chi^{-1}(A^+)$ is dense in $\tilde\Gamma$.
Thus the isomorphism of Theorem \ref{teo8} shows that we have proved that $\Upsilon\colon I\longrightarrow\liminv\ol_n^*$
is injective on a dense subset; the kernel must have empty interior.
\end{proof}

\begin{rem} Since $I\otimes_\Z\Z[\Delta]=\oplus_{\delta\in\Delta}I\delta$, formula \eqref{e:lambdahom} shows that $\Upsilon$ can be extended
to a homomorphism of $\tilde\L$-modules $I\otimes_\Z\Z[\Delta]\longrightarrow\liminv\calc_n\,$.  \end{rem}

\begin{prop} We have: $\liminv\ol_n^*/\liminv\calc_n\simeq\liminv\calu_n/\liminv\calc_n^1\,$.
\end{prop}

\begin{proof} Consider the commutative diagram
\[ \begin{CD} 1 @>>> \liminv\calc_n^1 @>>> \liminv\calu_n @>>> \liminv\calu_n/\liminv\calc_n^1 @>>> 1\\
& & @VV\alpha_1V @VV\alpha_2V @VV\alpha_3V \\
1 @>>> \liminv\calc_n @>>> \liminv\ol_n^* @>>> \liminv\ol_n^*/\liminv\calc_n @>>> 1\ .
\end{CD}\]
All vertical maps are injective and by \eqref{limK} the cokernel of $\alpha_2$ is $\F_\pr^*\,$. For $\delta\in\Delta$ one
has $\delta\omega_n=\Phi_{\chi(\delta)}(\omega_n)=\chi(\delta)\omega_nu_n$ for some $u_n\in\calu_n\,$. By the injectivity part of the
proof of Proposition \ref{p:upsilon}, $C_n=C_n^1\times\Upsilon(\Z[\Delta])$ and it follows that the cokernel of $\alpha_1$ is also $\F_\pr^*\,$.
\end{proof}

\subsection{Cyclotomic units and class groups} \label{s:cycunclssgrp} Let $F_n^+\subset F_n$ be the fixed field of the inertia group
$I_\infty(F_n/F)$. The extension $F_n^+/F$ is totally split at $\infty$ and ramified only above the prime $\pr$. We shall denote the
ring of $A$-integers of $F_n^+$ by $B_n^+$. Also, define $\cale_n$ and $\cale_n^1$ to be the closure respectively of $B_n^*\cap\ol_n^*$
and $B_n^*\cap\calu_n\,$.

We need to introduce a slight modification of the groups $\cala(L)$ of section \ref{s:classgroups}. For any finite extension $L/F$,
$\cala^{\infty}(L)$ will be the $p$-part of the class group of $A$-integers of $L$, so that, by class field theory,
$\cala^{\infty}(L)\simeq Gal(H(L)/L)$, where $H(L)$ is the maximal abelian unramified $p$-extension of $L$ which is totally split at
places dividing $\infty$. We shall use the shortening $\cala_n:=\cala^{\infty}(F_n^+)$.

Also, let $\mathcal{X}_n:=Gal(M(F_n^+)/F_n^+)$, where $M(L)$ is the maximal abelian $p$-extension of $L$ unramified outside $\pr$ and totally
split above $\infty$. As in the number field case, one has an exact sequence
\begin{equation} \label{e:seqiw}
\begin{CD} 1 @>>> \cale_n^1/\calc_n^1 @>>> \calu_n/\calc_n^1 @>>> \mathcal{X}_n @>>> \cala_n @>>> 1 \end{CD}
\end{equation}
coming from the following

\begin{prop}\label{prop6.12}
There is an isomorphism of Galois modules
$$\calu_n/\cale^1_n \simeq Gal(M(F_n^+)/H(F_n^+))\ .$$
\end{prop}

\begin{proof} This is a consequence of class field theory in characteristic $p>0$, as the analogous statement in the number field case: just
recall that the role of archimedean places is now played by the valuations above $\infty$. Under the class field theoretic identification of
idele classes $\idl_{F_n^+}/(F_n^+)^*$ with a dense subgroup of $Gal((F_n^+)^{ab}/F_n^+)$, one finds a surjection
\[ \prod_{\mathfrak P|\pr}\ol_{\mathfrak P}^* \sri Gal(M(F_n^+)/H(F_n^+)) \]
whose kernel contains the closure of
$$\prod_{\mathfrak P|\pr}\ol_{\mathfrak P}^*\cap (F_n^+)^*\prod_{w\nmid\pr}\ol_w^*\prod_{w|\infty}(F_{n,w}^+)^*=\iota_\pr((B_n^+)^*)$$
(where $\iota_\pr$ denotes the diagonal inclusion). Reasoning as in \cite[Lemma 13.5]{wash} one proves the proposition.
\end{proof}

Taking the projective limit of the sequence \eqref{e:seqiw}, we get
\begin{equation} \label{e:seqiwlim}
\begin{CD} 1 @>>> \cale_\infty^1/\calc_\infty^1 @>>> \calu_\infty/\calc_\infty^1 @>>> \mathcal{X}_\infty @>>> \cala_\infty @>>> 1\,.\end{CD}
\end{equation}

\begin{lem} The sequence \eqref{e:seqiwlim} is exact. \end{lem}

\begin{proof} Taking the projective limit of the short exact sequence
\[ \begin{CD} 1 @>>> \cale^1_n @>>> \calu_n @>>> Gal(M(F_n^+)/H(F_n^+)) @>>> 1 \end{CD} \]
we obtain
\[ \begin{CD} 1 @>>> \cale^1_\infty @>>> \calu_\infty @>>> Gal(M(\calf^+)/H(\calf^+)) @>>> {\liminv}^1\cale_n^1 \ , \end{CD} \]
where $M(\calf^+)$ and $H(\calf^+)$ are the maximal abelian $p$-extensions of $\calf^+$ totally split above $\infty$ and unramified
respectively outside the place above $\pr$ and everywhere.

To prove the lemma it is enough to show that ${\liminv}^1(\cale^1_n)=1$. By a well-known result in homological algebra, the functor
${\liminv}^1$ is trivial on projective systems satisfying the Mittag-Leffler condition. We recall that an inverse system $(B_n,d_n)$ enjoys
such property if for any $n$ the images of the transition maps $B_{n+m}\rightarrow B_n$ are the same for large $m$. So we are reduced to check
that this holds for the $\cale_n^1$'s with the norm maps.

Observe first that $\cale_n^1$ is a finitely generated $\Lambda_n$-module, thus noetherian because so is $\Lambda_n$. Consider
now $\cap_kImage(N_{n+k,n})$, where $N_{n+k,n}\colon\cale_{n+k}^1\rightarrow\cale^1_n$ is the norm map. This intersection is a
$\Lambda_n$-submodule of $\cale^1_n$, non-trivial because it contains the cyclotomic units. By noetherianity it is finitely generated,
hence there exists $l$ such that $Image(N_{n+k,n})$ is the same for $k\geqslant l$. Therefore $(\cale_n^1)$ satisfies the Mittag-Leffler property.
\end{proof}

The exact sequence \eqref{e:seqiwlim} lies at the heart of Iwasawa theory. Its terms are all $\Lambda$-modules and, in section
\ref{s:charideal}, we have shown how to associate a characteristic ideal to $\cala_\infty$ and its close relation with Stickelberger
elements. In a similar way, i.e., working on $\Z_p^d$-subextensions, one might approach a description of $\calx_\infty\,$, while, for the
first two terms of the sequence, the filtration of the $F_n^+$'s seems more natural (as the previous sections show).

Assume for example that the class number of $F$ is prime with $p$, then it is easy to see that $\cala_n=1$ for all $n$. Moreover, using
the fact that, by a theorem of Galovich and Rosen, the index of the cyclotomic units is equal to the class number (see
\cite[Theorem 16.12]{rosen}), one can prove that $\cale_n^1/\calc_n^1=1$ as well. These provide isomorphisms
\[ \calu_n/\calc_n^1\simeq\mathcal{X}_n \]
and
\[ \calu_\infty/\calc_\infty^1\simeq\mathcal{X}_\infty \ .\]

In general one expects a relation (at least at the level of $\Z_p^d$-subextensions, then a limit procedure should apply) between the
pro-characteristic ideal of $\cala_\infty$ and the (yet to be defined) analogous ideal for $\cale_\infty^1/\calc_\infty^1$ (the Stickelberger
element might be a first hint for the study of this relation). Consequently (because of the multiplicativity of characteristic ideals) an
equality of (yet to be defined) characteristic ideals of $\calx_\infty$ and of $\calu_\infty/\calc_\infty^1$ is expected as well. Any of
those two equalities can be considered as an instance of Iwasawa Main Conjecture for the setting we are working in.

\subsection{Bernoulli-Carlitz numbers} \label{s:berncarl} We go back to the subject of characteristic $p$ $L$-function. Let
$A^+\subset A$ be the subset of monic polynomials. The Carlitz zeta function \index{Carlitz-Goss zeta function} is defined
$$\zeta_A(k):=\sum_{a\in A^+}\frac{1}{a^k}$$
for $k\in\N$.

Recall that the Carlitz module corresponds to a lattice $\xi A\subset\dc$ and can be constructed via the Carlitz exponential
$e_C(z):=z\prod_{a\in A'}(1-z\xi^{-1}a^{-1})$ (where $A'$ denotes $A-\{0\}$). Rearranging summands in the equality
$$\frac{1}{e_C(z)}=\dlog(e_C(z))=\sum_{a\in A}\frac{1}{z-\xi a}=\frac{1}{z}-\sum_{a\in A'}\sum_{k=1}^\infty\frac{z^{k-1}}{(\xi a)^k}$$
(and using $A'=\F_q^*\times A^+$) one gets the well-known formula
\begin{equation} \label{zetacarlitz} \frac{1}{e_C(z)}=\frac{1}{z}+\sum_{n=1}^\infty\frac{\zeta_A(n(q-1))}{\xi^{n(q-1)}}z^{n(q-1)-1}\ . \end{equation}

From section \ref{s:cycunit} it follows that for any $a,b\in A-\pr$, the function $\frac{\Phi_{a}(x)}{\Phi_{b}(x)}$ is an $\caln$-invariant power
series, associated with
\begin{equation} \label{colcycunit} c(a,b):=\frac{\Phi_{a}(\omega)}{\Phi_{b}(\omega)}=
\left(\frac{\Phi_{a}(\omega_n)}{\Phi_{b}(\omega_n)}\right)_n\in\liminv\mathcal{O}_n^*\ .\end{equation}

\begin{thm} \label{main} The $k$th Coates-Wiles homomorphism applied to $c(a,b)$ is equal to:
\[ \delta_k(c(a,b))=\left\{ \begin{array}{ll}
0 & {\rm if}\ k\not\equiv 0 \pmod{q-1} \\
(a^k-b^k)\frac{\zeta_A(k)}{\xi^k} & {\rm if}\ k\equiv 0 \pmod{q-1}
\end{array} \right. \ .\]
\end{thm}

We remark that the condition $k=n(q-1)$ here is the analogue of $k$ being an even integer in the classical setting (since $q-1=|\F_q^*|$
just as $2=|\Z^*|$).

\begin{proof} Observe that \eqref{colcycunit} amounts to giving the Coleman power series $Col_{c(a,b)}$.\\
Let $\lambda$ be the Carlitz logarithm, i.e., $\lambda\in F\{\{\tau\}\}$ is the element uniquely determined by $e_C\circ\lambda=1$.
Then $\Phi_a(x)=e_C(a\lambda(x))$ and by \eqref{colcycunit} and \eqref{zetacarlitz} one gets
\[ \begin{array}{ll} \dlog Col_{c(a,b)}(x) & =\frac{a}{\Phi_a(x)}-\frac{b}{\Phi_b(x)} \\
\ & = \sum_{n\geqslant 1}(a^{n(q-1)}-b^{n(q-1)})\frac{\zeta_A(n(q-1))}{\xi^{n(q-1)}}\lambda(x)^{n(q-1)-1}\ . \end{array}\]
Since $\lambda(e_C(x))=x$, we get
\begin{equation} \label{eq:cwcycl} (\dlog Col_{c(a,b)})(e_C(x))=
\sum_{n\geqslant 1}(a^{n(q-1)}-b^{n(q-1)})\frac{\zeta_A(n(q-1))}{\xi^{n(q-1)}}x^{n(q-1)-1}\end{equation}
and the theorem follows comparing \eqref{eq:cwcycl} with \eqref{eq:deltaexp}.
\end{proof}

\begin{rem} As already known to Carlitz, $\zeta_A(k)\xi^{-k}$ is in $F$ when $q-1$ divides $k$. Note that by a theorem of Wade,
$\xi\in F_\infty$ is transcendental over $F$. Furthermore, Jing Yu \cite{jyu} proved that $\zeta_A(k)$ for all $k\in\N$ and
$\zeta_A(k)\xi^{-k}$ for $k$ ``odd'' (i.e., not divisible by $q-1$) are transcendental over $F$. \end{rem}

Theorem \ref{main} can be restated in terms of the Bernoulli-Carlitz numbers $BC_k$ \cite[Definition 9.2.1]{goss}. They can be defined by
$$\frac{1}{e_C(z)}=\sum_{n\geqslant 0}\frac{BC_n}{\Pi(n)}z^{n-1}$$ \index{Bernoulli-Carlitz numbers}
(where $\Pi(n)$ is a function field analogue of the classical factorial $n!$); in particular $BC_n=0$ when $n\not\equiv 0 \pmod{q-1}$.
Then Theorem \ref{main} becomes
\begin{equation} \label{e:BC} \delta_k(c(a,b))=(a^k-b^k)\frac{BC_k}{\Pi(k)}\ .\end{equation}

Theorem \ref{main} and formula \eqref{e:BC} can be seen as extending a result by Okada, who in \cite{Okada} obtained
the ratios $\frac{BC_k}{\Pi(k)}$ (for
$k=1,\dots,q^{\deg(\pr)}-2$) as images of cyclotomic units under the Kummer homomorphisms (which are essentially a less refined version of
the Coates-Wiles homomorphisms). From here one proves that the non-triviality of an isotypic component of $\cala_1$ implies the divisibility
of the corresponding ``even'' Bernoulli-Carlitz number by $\pr$: we refer to \cite[\S8.20]{goss} for an account. As already mentioned,
Shu \cite{shu} generalized Okada's work to any $F$ (but with the assumption $\deg(\infty)=1$): it might be interesting to extend Theorem
\ref{main} to a ``Coates-Wiles homomorphism'' version of her results.

\subsection{Interpolation?} In the classical setting of cyclotomic number fields, the analogue of the formula in Theorem \ref{main} can be used
as a key step in the construction of the Kubota-Leopoldt zeta function (see e.g. \cite{cosu}). Hence it is natural to wonder if something like
it holds in our function field case. For now we have no answer and can only offer some vague speculation.

As mentioned in section \ref{s:gosszeta}, Goss found a way to extend the domain of $\zeta_A$ from $\N$ to $S_\infty$. \index{Carlitz-Goss
zeta function} He also considered the analogue of the $p$-adic domain and defined it to be $\cp^*\times S_\pr$, with
$S_\pr:=\Z_p\times\Z/(q^{\deg(\pr)}-1)$ (observe that $\cp^*\times S_\pr$ is the $\cp$-valued dual of $F_\pr^*$). Then functions like
$\zeta_A$ enjoy also a $\pr$-adic life: for example, letting $\pi_v\in A^+$ be a uniformizer for a place $v$, $\zeta_{A,\pr}$ is defined
on $\cp^*\times S_\pr$ by
$$\zeta_{A,\pr}(s):=\prod_{v\nmid\pr\infty}(1-\pi_v^{-s})^{-1}\ ,$$
at least where the product converges.

The ring $\Z$ embeds discretely in $S_\infty$ and has dense image in $1\times S_\pr$. So Theorem \ref{main} seems to suggest interpolation
of $\zeta_{A,\pr}$ on $1\times S_\pr$. Another clue in this direction is the fact that $S_\pr$ is the ``dual'' of $\Gamma$, just as $\Z_p$
is the ``dual'' of $Gal(\Q(\zeta_{p^\infty})/\Q)$ (a strengthening of this interpretation has been recently provided by the main result
of \cite{jeong}).\bigskip

\noindent {\bf Acknowledgments.} We thank Sangtae Jeong, King Fai Lai, Jing Long Hoelscher, Ki-Seng Tan, Dinesh Thakur, Fabien Trihan for
useful conversations and Bruno Angl\`es for pointing out a mistake in the first version of this paper. F. Bars and I. Longhi thank CRM
for providing a nice environment to complete work on this paper.


\vspace{.5truecm}
\noindent Andrea Bandini\\
Dipartimento di Matematica, Universit\`a degli Studi di Parma\\
Parco Area delle Scienze, 53/A - 43124 Parma (PR), Italy\\
andrea.bandini@unipr.it\\

\noindent Francesc Bars\\
Departament Matem\`atiques, Edif. C, Universitat Aut\`onoma de Barcelona\\
08193 Bellaterra (Barcelona), Catalonia\\
francesc@mat.uab.cat\\

\noindent Ignazio Longhi\\
Department of Mathematical Sciences, Xi'an Jiaotong-Liverpool
University\\ 111 Ren Ai Road, Dushu Lake Higher Education Town
Suzhou Industrial Park, Suzhou, Jiangsu 215123, China\\
Ignazio.Longhi@xjtlu.edu.cn\\
\end{document}